\documentclass[reqno,10pt,centertags,draft]{amsart}
\usepackage{amsmath,amsthm,amscd,amssymb,latexsym,upref,stmaryrd}
\usepackage{amsfonts,mathrsfs}

\newcommand{\R}{{\bbR}}
\newcommand{\N}{{\mathbb N}}

\newcommand{\C}{{\mathbb C}}

\newcommand{\bbC}{{\mathbb{C}}}

\newcommand{\bbN}{{\mathbb{N}}}

\newcommand{\bbR}{{\mathbb{R}}}

\newcommand{\cA}{{\mathcal A}}
\newcommand{\cB}{{\mathcal B}}
\newcommand{\cC}{{\mathcal C}}
\newcommand{\cD}{{\mathcal D}}

\newcommand{\cH}{{\mathcal H}}

\newcommand{\cK}{{\mathcal K}}

\newcommand{\cP}{{\mathcal P}}

\newcommand{\cR}{{\mathcal R}}

\DeclareMathOperator{\supp}{supp}

\DeclareMathOperator{\dom}{dom}

\DeclareMathOperator{\tr}{tr}

\DeclareMathOperator*{\slim}{s-lim}

\renewcommand{\Im}{\text{\rm Im}}
\renewcommand{\ln}{\text{\rm ln}}

\newcommand{\loc}{\text{\rm{loc}}}

\newcommand{\beq}{\begin{equation}}
\newcommand{\enq}{\end{equation}}

\newcommand{\sgn}{{\textrm{sgn}}}

\newcommand{\no}{\notag}
\newcommand{\lb}{\label}
\newcommand{\f}{\frac}

\newcommand{\ol}{\overline}

\newcommand{\wti}{\widetilde}

\newcommand{\hatt}{\widehat}
\newcommand{\bi}{\bibitem}

\let\geq\geqslant
\let\leq\leqslant

\makeatletter
\def\theequation{\@arabic\c@equation}

\allowdisplaybreaks 
\numberwithin{equation}{section}

\newtheorem{theorem}{Theorem}[section]

\newtheorem{lemma}[theorem]{Lemma}
\newtheorem{corollary}[theorem]{Corollary}
\newtheorem{hypothesis}[theorem]{Hypothesis}

\theoremstyle{remark}
\newtheorem{remark}[theorem]{Remark}
\newtheorem{definition}[theorem]{Definition}

\begin{document}

\title[Weak Convergence of Spectral Shift Functions]{An Abstract Approach to Weak Convergence of Spectral Shift Functions and Applications to Multi-Dimensional 
Schr\"odinger Operators}

\author[F.\ Gesztesy and R.\ Nichols]{Fritz Gesztesy and Roger Nichols}
\address{Department of Mathematics,
University of Missouri, Columbia, MO 65211, USA}
\email{gesztesyf@missouri.edu}
\urladdr{http://www.math.missouri.edu/personnel/faculty/gesztesyf.html}
\address{Department of Mathematics,
University of Missouri, Columbia, MO 65211, USA}
\email{nicholsrog@missouri.edu}

\dedicatory{Dedicated to the memory of Nigel J.\ Kalton (1946--2010)}
\date{\today}

\begin{abstract} 
We study the manner in which a sequence of spectral shift functions $\xi(\cdot;H_j,H_{0,j})$ 
associated with abstract pairs of 
self-adjoint operators $(H_j, H_{0,j})$ in Hilbert spaces $\cH_j$, $j\in\bbN$, converge to a limiting spectral shift function $\xi(\cdot;H,H_0)$ associated with a pair $(H,H_0)$ in the limiting Hilbert space $\cH$ as $j\to\infty$ (mimicking the infinite volume limit in concrete applications to multi-dimensional Schr\"odinger operators). 

Our techniques rely on a Fredholm determinant approach combined with certain measure theoretic facts. In particular, we show that prior vague convergence results for spectral 
shift functions in the literature actually extend to the notion of weak convergence. More precisely, in the concrete case of multi-dimensional Schr\"odinger operators on a 
sequence of domains $\Omega_j$ exhausting $\bbR^n$ as $j\to\infty$, we extend the convergence of associated spectral shift functions from vague to weak convergence and also from Dirichlet boundary conditions to more general self-adjoint boundary conditions 
on $\partial\Omega_j$. 
\end{abstract}

\maketitle

\section{Introduction}  \lb{s1}

We are interested in the manner in which a sequence of spectral shift functions for 
abstract pairs of self-adjoint operators $(H_j, H_{0,j})$ in 
Hilbert spaces $\cH_j$, $j\in\bbN$, converge to a limiting spectral shift function 
associated with a pair $(H,H_0)$ in a limiting 
Hilbert space $\cH$ as $j \to\infty$ (mimicking the infinite volume limit in concrete 
situations). As a concrete application we explicitly treat the case of multi-dimensional 
Schr\"odinger on bounded domains $\Omega_j \subset \bbR^n$, $n\in\bbN$, exhausting 
$\bbR^n$ as $j\to\infty$, with various boundary conditions on $\partial\Omega_j$, 
$j\in\bbN$ (we primarily focus on the cases $1 \leq n \leq 3$).  

An exhaustive treatment of the special one-dimensional case appeared in \cite{GN11}. 

Before we focus on the abstract situation discussed in this paper, it is 
appropriate to briefly survey the known results in this area. Consider self-adjoint 
Schr\"odinger operators $H_j$ and $H_{0,j}$ 
in $L^2((-j,j)^n; d^n x)$, $n\in\bbN$, $n\geq 2$, generated 
by the differential expression $-\Delta + V$ and $-\Delta$ on $(-j,j)^n$, respectively, 
with Dirichlet boundary conditions on $\partial (-j,j)^n$, where 
$0 \leq V \in L^\infty(\bbR^n; d^nx)$ is of fixed compact support in $(-j,j)^n$, 
real-valued, and nonzero a.e.\ Denoting by $\xi\big(\lambda; H_j,H_{0,j}\big)$ for 
a.e.\ $\lambda \in \R$, the spectral shift function associated with the pair 
$(H_j, H_{0,j})$ (cf.\ \cite[Ch.\ 8]{Ya92}), normalized to be zero in a neighborhood of 
$-\infty$, Kirsch \cite{Ki87} showed in 1987 that (perhaps, somewhat unexpectedly) for any $\lambda > 0$, 
\begin{equation}
\sup_{j\in\bbN} \big|\xi\big(\lambda; H_j,H_{0,j}\big)\big| = \infty.    \lb{1.1}
\end{equation}
Moreover, denoting by $H$ and $H_0$ the corresponding self-adjoint Schr\"odinger operators in $L^2(\bbR^n; d^n x)$ generated by the differential expression 
$-\Delta + V$ and $-\Delta$ on $\R^n$, respectively, one cannot expect pointwise 
a.e.\ convergence (or convergence in measure) 
of $\xi\big(\cdot; H_j,H_{0,j}\big)$ to $\xi\big(\lambda; H,H_0\big)$ in the infinite volume limit $j\to\infty$ by the following elementary argument: For a.e.\ $\lambda > 0$, 
$\xi\big(\lambda; H,H_0\big)$ is a continuous function with respect to $\lambda$, 
related to the determinant of the underlying fixed energy scattering matrix. Yet 
$\xi\big(\cdot; H_j,H_{0,j}\big)$, as a difference of eigenvalue counting functions corresponding to the number of eigenvalues (counting multiplicity) of $H_j$ and 
$H_{0,j}$, respectively, is integer-valued and hence cannot possibly converge to a non-constant continuous function as $j\to\infty$. In particular, this argument applies to the 
one-dimensional context (in which case $\xi\big(\lambda; H,H_0\big) \to 0$ as 
$\lambda\to\infty$).

Having ruled out pointwise a.e.\ convergence of spectral shift functions in the infinite 
volume limit $j\to\infty$ in all space dimensions, it becomes clear that one has to invoke the concept of certain  generalized limits. Indeed, in 1995, Geisler, Kostrykin, and 
Schrader \cite{GKS95} proved for potentials $V \in \ell^1(L^2(\bbR^3; d^3x))$ (a 
Birman--Solomyak space, cf., e.g., \cite[Ch.\ 4]{Si05}) that for all $\lambda \in \R$,
\begin{equation}
\lim_{j\to\infty} \int_{(-\infty,\lambda]} \xi\big(\lambda'; H_j,H_{0,j}\big) d\lambda'
= \int_{(-\infty,\lambda]} \xi\big(\lambda'; H,H_0\big) d\lambda'.    \lb{1.2}
\end{equation}
Since $H_j$ and $H_{0,j}$ are bounded from below uniformly with respect to 
$j\in\bbN$, the limiting relation \eqref{1.2} involving distribution functions of the 
spectral shift measures is equivalent to vague convergence of the latter as observed in 
\cite[Prop.\ 4.3]{HLMW01}.  

In the one-dimensional half-line context, Borovyk and Makarov \cite{BM09} (see also 
Borovyk \cite{Bo08}) proved in 2009 that for potentials $V \in L^1((0,\infty);(1+ |x|)dx)$ 
real-valued, and denoting by $H_R$ the self-adjoint Schr\"odinger operator in 
$L^2((0,R);dx)$ and $H$ the corresponding self-adjoint Schr\"odinger operator in 
$L^2((0,\infty);dx)$, both with Dirichlet boundary conditions (and otherwise maximally defined or defined in terms of quadratic forms), and analogously for $H_{0,R}$ and 
$H_0$ in the unperturbed case $V=0$, the following vague limit holds: 
\begin{align}
\begin{split}
& \, \text{For any $g \in C_0(\bbR)$,} \\ 
& \lim_{R\to\infty} \int_{\R} \xi\big(\lambda; H_R,H_{0,R}\big) d\lambda \, g(\lambda)
= \int_{\R} \xi\big(\lambda; H,H_0\big) d\lambda \, g(\lambda).   \lb{1.3}
\end{split} 
\end{align}
In addition, they proved that the following Ces{\` a}ro limit, 
\begin{equation}
\lim_{R\to\infty} \f{1}{R} \int_0^R \xi\big(\lambda; H_r,H_{0,r}\big) dr = 
\xi\big(\lambda; H,H_0\big), \quad 
\lambda \in \bbR\backslash\big(\sigma_{\rm d} (H) \cup \{0\}\big)    \lb{1.4}
\end{equation}
exists (and the limit in \eqref{1.4} extends to $\lambda = 0$ if $H$ has no zero energy resonance). 

Returning to the case of multi-dimensional boxes $[-R,R]^n$, Hislop and 
M\"uller \cite{HM10} (see also \cite{HM10a}) proved a result going somewhat beyond 
vague convergence in 2010. More precisely, assuming a real-valued background 
potential $V^{(0)}$ satisfying $V_-^{(0)} \in K(\R^n)$, $V_+^{(0)} \in K_{\rm loc}(\R^n)$ 
and a potential $0 \leq V \in K_{\rm loc}(\R^n)$, $\supp (V)$\,compact (cf.\ \cite{AS82}, 
\cite{Si82} for the definition of (local) Kato classes), they show that 
\begin{align}
& \, \text{For any $f \in C_0(\bbR)$, and for any $f = \chi_{J}$, $J\subset\R$ a finite interval,}    \no \\ 
& \lim_{R\to\infty} \int_{\R} 
\xi\big(\lambda; H_{0,R} + V^{(0)} +V,H_{0,R}+ V^{(0)}\big) d\lambda \, f(\lambda)  
\lb{1.5} \\
& \quad = 
\int_{\R} \xi\big(\lambda; H_0 + V^{(0)} +V,H_0+ V^{(0)}\big) d\lambda \, f(\lambda).   \no
\end{align}
In addition, they derived a weaker version than the Ces\'aro limit in (1.4) in the multi-dimensional context. More precisely, they proved that for every sequence of lengths 
$\{L_j\}_{j\in\N} \subset (0,\infty)$ with $\lim_{j\to\infty} L_j = \infty$, there exists a subsequence $\{j_k\}_{k\in\N} \subset \N$ with
$\lim_{k\to\infty} j_k = \infty$, such that for every subsequence 
$\{k_{\ell}\}_{\ell\in\N} \subset \N$ with $\lim_{\ell\to\infty} k_{\ell} = \infty$,
\begin{align}
\begin{split} 
& \lim_{L\to\infty} \f{1}{L} \sum_{\ell=1}^L 
\xi\big(\lambda; H_{L_{j_{k_\ell}}}^{(0)} + V^{(0)} +V, 
H_{L_{j_{k_\ell}}}^{(0)} + V^{(0)}\big)    \lb{1.6} \\
& \quad \leq \xi\big(\lambda; H_0 + V^{(0)} +V,H_0 + V^{(0)}\big)
\end{split} 
\end{align}
for (Lebesgue) a.e.\ $\lambda \in \R$.

Before describing our results we should mention that spectral shift functions feature prominently in the context of eigenvalue counting functions and hence in the 
context of the integrated density of states. We refer, for instance, to \cite{CHKN02}, 
\cite{CHN01}, \cite{HK02}, \cite{HKNSV06}, \cite{HS02}, \cite{Ko99}, \cite{KS99}, 
\cite{KS00}, \cite{Na01}, and the references cited therein. For bounds on the 
spectral shift function we refer to \cite{CHN01}, \cite{HKNSV06}, \cite{HS02}, and 
\cite{So93}. 

In Section \ref{s2} we collect basic properties of spectral shift functions used in the bulk of this paper. In Section \ref{s3} we prove our principal abstract result, the convergence of a sequence of spectral shift functions $\xi(\cdot;H_j,H_{0,j})$ associated with pairs of self-adjoint operators $(H_j, H_{0,j})$ in Hilbert spaces $\cH_j$, $j\in\bbN$, to the limiting 
spectral shift function $\xi(\cdot;H,H_0)$ associated with the pair $(H,H_0)$ in a limiting 
Hilbert space $\cH$ as $j \to\infty$, thus mimicking the infinite volume limit in concrete 
situations. Finally, in Section \ref{s4} we provide detailed applications to Schr\"odinger 
operators in dimensions $n =1,2,3$ in the case of Dirichlet boundary conditions and sketch 
extensions to higher dimensions $n \geq 4$ and Robin boundary conditions.  
 
Finally, we briefly summarize some of the notation used in this paper: Let $\cH$ be a
separable complex Hilbert space, $(\cdot,\cdot)_{\cH}$ the scalar product in $\cH$
(linear in the second argument), and $I_{\cH}$ the identity operator in $\cH$.
Next, let $T$ be a linear operator mapping (a subspace of) a
Hilbert space into another, with $\dom(T)$ and $\ker(T)$ denoting the
domain and kernel (i.e., null space) of $T$. 
The closure of a closable operator $S$ is denoted by $\ol S$. 
The resolvent set, spectrum, essential spectrum, discrete spectrum, and resolvent set 
of a closed linear operator in $\cH$ will be denoted by $\rho(\cdot)$, $\sigma(\cdot)$, $\sigma_{\rm ess}(\cdot)$, $\sigma_{\rm d}(\cdot)$, and $\rho(\cdot)$, respectively. 
The Banach space of bounded (resp., compact) linear operators on $\cH$ is
denoted by $\cB(\cH)$ (resp., $\cB_{\infty}(\cH)$). The corresponding 
$\ell^p$-based trace ideals will be denoted by $\cB_p (\cH)$, $p>0$. 
The trace of trace class operators in $\cH$ is denoted by ${\tr}_{\cH}(\cdot)$, modified Fredholm determinants are abbreviated by ${\det}_{p,\cH}(I_{\cH} + \cdot)$, 
$p\in\bbN$ (the subscript $p$ being omitted in the trace class case $p=1$).

The form sum (resp.\ difference) of two self-adjoint operators $A$ and $W$ will be denoted by $A +_q W$ (resp., $A -_q W = A +_q (-W)$).

\section{Basic Facts on Spectral Shift Functions}  \lb{s2}

In this preparatory section we succinctly summarize properties of the spectral shift function as needed in the bulk of this paper (for details on this material we refer to \cite{BY93}, 
\cite[Ch.\ 8]{Ya92}, \cite{Ya07}, \cite[Sect.\ 0.9, Chs.\ 4, 5, 9]{Ya10}).

We start with the following basic assumptions:

\begin{hypothesis} \lb{h2.1}
Suppose $A$ and $B$ are self-adjoint operators in $\cH$ with $A$ bounded from below. \\
$(i)$ Assume that $B$ can be written as the form sum of $A$ and 
a self-adjoint operator $W$ in $\cH$
\begin{equation}
B = A +_q W,     \lb{2.1}
\end{equation}
where $W$ can be factorized into
\begin{equation}
W = W_1^* W_2,    \lb{2.2}
\end{equation}
such that 
\begin{equation}
\dom(W_j) \supseteq \dom\big(|A|^{1/2}\big), \quad j =1,2,   \lb{2.2a}
\end{equation}
$(ii)$ Suppose that for some $($and hence for all\,$)$ $z \in \rho(A)$,
\begin{equation}
\ol{W_2 (A - z I_{\cH})^{-1} W_1^*} \in \cB_1(\cH),    \lb{2.3}
\end{equation}
and that 
\begin{equation}
\lim_{z \downarrow -\infty} \big\|\ol{W_2 (A - z I_{\cH})^{-1} W_1^*}\big\|_{\cB_1(\cH)} = 0. 
\lb{2.3a}
\end{equation}
$(iii)$ In addition, we suppose that for some $($and hence for all\,$)$ 
$z \in \rho(B) \cap \rho(A)$,
\begin{equation}
\big[(B - z I_{\cH})^{-1} - (A - z I_{\cH})^{-1}\big] \in \cB_1(\cH).  \lb{2.4}
\end{equation}
\end{hypothesis}

\medskip

Extensions of Hypothesis \ref{h2.1} where $\cB_1(\cH)$ in \eqref{2.3} is replaced by 
$\cB_p(\cH)$ for some $p\in\bbN$, and/or the resolvents in \eqref{2.4} will be replaced by 
higher powers of resolvents, will be discussed a bit later.

Given Hypothesis \ref{h2.1}\,$(i)$, one observes that 
\begin{equation}
\dom\big(|B|^{1/2}\big) = \dom\big(|A|^{1/2}\big),    \lb{2.5} 
\end{equation}
and that the resolvent of $B$ can be written as (cf., e.g., the detailed discussion in 
\cite{GLMZ05} and the references therein)
\begin{align}
(B - z I_{\cH})^{-1} &= (A - z I_{\cH})^{-1}    \no \\ 
& \quad - \ol{(A - z I_{\cH})^{-1} W_1^*} 
\big[I_{\cH} + \ol{W_2 (A - z I_{\cH})^{-1} W_1^*}\big]^{-1} W_2 (A - z I_{\cH})^{-1}, 
\no\\
& \hspace*{6.7cm} z \in \rho(B) \cap \rho(A).    \lb{2.6} 
\end{align}
In particular, $B$ is bounded from below in $\cH$.

Moreover, assuming the full Hypothesis \ref{h2.1} one infers that (cf.\ \cite{GZ11}) 
\begin{align}
& {\tr}_{\cH}\big((B - z I_{\cH})^{-1} - (A - z I_{\cH})^{-1}\big)   \no \\
& \quad 
= - \f{d}{dz} \ln\Big({\det}_{\cH}\Big(\ol{(B - z I_{\cH})^{1/2}(A - z I_{\cH})^{-1}
	(B - z I_{\cH})^{1/2}}\Big)\Big)    \no \\
&\quad =  - \f{d}{dz} \ln\Big({\det}_{\cH}\Big(I_{\cH} +\ol{W_2 (A - z I_{\cH})^{-1} 
W_1^*}\Big)\Big), \quad z \in \rho(B) \cap \rho(A).     \lb{2.7}
\end{align}

In addition, Hypothesis \ref{h2.1} guarantees the existence of the real-valued spectral shift function $\xi(\cdot; B,A)$ satisfying
\begin{align}
{\tr}_{\cH}\big((B - z I_{\cH})^{-1} - (A - z I_{\cH})^{-1}\big) 
= - \int_{\bbR} \f{\xi(\lambda; B,A) \, d\lambda}{(\lambda - z)^2},  \quad 
z \in \rho(B) \cap \rho(A)),        \lb{2.8}  
\end{align}
and 
\begin{align}
& \xi(\lambda; B,A) = 0, \quad \lambda < \inf(\sigma(B), \sigma(A)),    \lb{2.9} \\
& \xi(\cdot; B,A) \in L^1\big(\bbR; (\lambda^2 + 1)^{-1} d\lambda\big). 
\lb{2.10}
\end{align}
Moreover, for a large class of functions $f$ (e.g., any $f$ s.t.\ 
$\hatt f (\cdot) \in L^1(\bbR; (|p| + 1) dp)$) one infers that 
$[f(B) - f(A)] \in \cB_1(\cH)$ and 
\begin{equation}
{\tr}_{\cH} (f(B) - f(A)) = \int_{\bbR} f'(\lambda) \, \xi(\lambda;B,A) \, d\lambda. 
\lb{2.11}
\end{equation}
This applies, in particular, to powers of the resolvent, where 
$f(\cdot) = (\cdot -z)^{-n}$, $n \in \bbN$, and we refer to \cite[Ch.\ 8]{Ya92} 
for more details. 

Throughout this manuscript we assume that the normalization \eqref{2.9} is applied. 

For subsequent purpose we summarize the results \eqref{2.7}, \eqref{2.8} as 
\begin{align}
& {\tr}_{\cH}\big((B - z I_{\cH})^{-1} - (A - z I_{\cH})^{-1}\big) 
= - \int_{\bbR} \f{\xi(\lambda; B,A) \, d\lambda}{(\lambda - z)^2}        \no \\
&\quad =  - \f{d}{dz} \ln\Big({\det}_{\cH}\Big(I_{\cH} +\ol{W_2 (A - z I_{\cH})^{-1} 
W_1^*}\Big)\Big), \quad z \in \rho(B) \cap \rho(A).     \lb{2.11A}
\end{align}
Here ${\det}_{\cH}(\cdot)$ denotes the Fredholm determinant (cf.\ \cite[Ch.\ IV]{GK69},
\cite{Si77}, \cite[Ch.\ 3]{Si05}). 

We also note the following monotonicity result: If 
\begin{align}
\begin{split}
& B \geq A \, \text{ (resp., $B \leq A$) in the sense of quadratic forms, then } \\
& \quad \xi(\lambda; B,A) \geq 0 \, \text{ (resp., $\xi(\lambda; B,A) \leq 0$).}
\lb{2.11a}
\end{split} 
\end{align}
Here, $B \geq A$ is meant in the sense of quadratic forms, that is, 
\begin{align}
& \dom \big(|A|^{1/2}\big) \supseteq \dom \big(|B|^{1/2}\big) \, \text{ and } 
\lb{2.11b} \\ 
& \quad \big(|B|^{1/2} f, \sgn(B) |B|^{1/2} f\big)_{\cH} \geq 
\big(|A|^{1/2} f, \sgn(A) |A|^{1/2} f\big)_{\cH}, \quad 
f \in \dom \big(|B|^{1/2}\big).    \no 
\end{align}

Next, suppose that the self-adjoint operator $C$ in $\cH$ can be written as the form sum of $B$ and a self-adjoint operator $Q$ in $\cH$, $C = B +_q Q$, where $Q$ can be factored as $Q=Q_1 Q_2$, with $Q, Q_1, Q_2$ satisfying the assumptions of $W, W_1, W_2$ in Hypotheses \ref{h2.1}. Then the formula
\begin{equation}
\xi(\lambda;C,A) = \xi(\lambda;C,B) + \xi(\lambda;B,A) \, 
\text{ for a.e.\ $\lambda \in \bbR$,}     \lb{2.11c}
\end{equation}
holds. 

Finally, we mention the connection between $\xi(\cdot; B,A)$ and the Fredholm determinant in \eqref{2.7},
\begin{equation}
\xi(\lambda; B,A) = \pi^{-1} \lim_{\varepsilon\downarrow 0} \Im\Big(\ln\Big(I_{\cH}
+ \ol{W_2 (A - (\lambda + i \varepsilon) I_{\cH})^{-1} W_1^*}\Big)\Big) \, 
\text{ for a.e.\ $\lambda \in \bbR$},    \lb{2.12}
\end{equation}
choosing the branch of $\ln({\det}_{\cH}(\cdot))$ on $\bbC_+$ such that 
\begin{equation}
\lim_{|\Im(z)| \to \infty} \ln\Big({\det}_{\cH}\Big(I_{\cH} 
+ \ol{W_2 (A - z I_{\cH})^{-1} W_1^*}\Big)\Big) = 0.    \lb{2.13} 
\end{equation}

For applications to multi-dimensional Schr\"odinger operators the framework in 
Hypothesis \ref{h2.1} is not sufficiently general and the trace class assumption, 
\eqref{2.3}, needs to be replaced by a weaker Hilbert--Schmidt-type  
hypothesis as detailed next: 

\begin{hypothesis} \lb{h2.2}
Suppose the assumptions made in Hypothesis \ref{h2.1} with the exception of the trace class hypothesis \eqref{2.3}. \\ 
$(iv)$ Assume that for some $($and hence for all\,$)$ $z \in \rho(A)$,
\begin{equation}
\ol{W_2 (A - z I_{\cH})^{-1} W_1^*} \in \cB_2(\cH).   \lb{2.19}
\end{equation}
and that 
\begin{equation}
\lim_{z \downarrow -\infty} \big\|\ol{W_2 (A - z I_{\cH})^{-1} W_1^*}\big\|_{\cB_2(\cH)} = 0. 
\lb{2.20}
\end{equation}
$(v)$ Suppose that
\begin{equation}
{\tr}_{\cH} \big(\ol{(A - z I_{\cH})^{-1} W (A - z I_{\cH})^{-1}}\big) = 
\eta' (z), \quad z \in \rho(A),    \lb{2.21}
\end{equation} 
where $\eta(\cdot)$ has normal limits, denoted by $\eta(\lambda + i0)$, 
for a.e.\ $\lambda \in \bbR$.
\end{hypothesis}

Then \eqref{2.5}, \eqref{2.6}, \eqref{2.8}--\eqref{2.11c} remain valid, but \eqref{2.7}, 
\eqref{2.12}, and \eqref{2.13} need to be amended as follows:

\begin{theorem} \lb{t2.3}
Assume Hypothesis \ref{h2.2}. Then
\begin{align}
\begin{split} 
\xi(\lambda; B, A) &= \pi^{-1} \Im\big(\ln\big(
{\det}_{2,\cH}\big(I_{\cH} + \ol{W_2 (A - (\lambda + i 0) I_{\cH})^{-1} W_1^*}\big)\big)\big)   \lb{2.21a} \\
& \quad + \pi^{-1} \Im(\eta(\lambda + i0)) + c \, \text{ for a.e.\ $\lambda \in \bbR$.} 
\end{split} 
\end{align}
Here $c\in\bbR$ has to be chosen in accordance with the normalization \eqref{2.9}. 
\end{theorem} 
\begin{proof}
First, one notes that 
\begin{align}
& \big|{\tr}_{\cH} \big(\ol{(A - z I_{\cH})^{-1} W (A - z I_{\cH})^{-1}}\big)\big| 
= \big|{\tr}_{\cH} \big(\ol{W_2 (A - z I_{\cH})^{-2} W_1^*}\big)\big|   \no \\
& \quad \leq \big\|\big(\ol{(A - z I_{\cH})^{-1} W (A - z I_{\cH})^{-1}}
\big\|_{\cB_1(\cH)}   \no \\
& \quad \leq |z|^{-1} \big\|W_1 (A - {\ol z} I_{\cH})^{-1/2}\big\|_{\cB(\cH)}
\big\|W_2 (A - z I_{\cH})^{-1/2}\big\|_{\cB(\cH)}    \no \\
& \; \underset{z \downarrow - \infty}{=} C |z|^{-1}.     \lb{2.22}
\end{align}
Next, one recalls that 
\begin{align}
& (B -z I_{\cH})^{-1} - (A - z I_{\cH})^{-1} + 
(A - z I_{\cH})^{-1} W (A - z I_{\cH})^{-1}     \no \\ 
& \quad = (A - z I_{\cH})^{-1} W_1^* 
\big[I_{\cH} + \ol{W_2 (A - z I_{\cH})^{-1} W_1^*}\big]^{-1}  \no \\
& \qquad \times \ol{W_2 (A - z I_{\cH})^{-1} W_1^*} W_2 (A - z I_{\cH})^{-1}, \quad 
z \in \rho(B) \cap \rho(A),      \lb{2.23}
\end{align}
and hence (cf.\ \cite[Sect.\ 1.7]{Ya92}) that 
\begin{align}
& {\tr}_{\cH} \big((B -z I_{\cH})^{-1} - (A - z I_{\cH})^{-1} + 
\ol{(A - z I_{\cH})^{-1} W (A - z I_{\cH})^{-1}}\big)     \no \\ 
& \quad = 
{\tr}_{\cH} \Big(\big[I_{\cH} + \ol{W_2 (A - z I_{\cH})^{-1} W_1^*}\big]^{-1} 
 \ol{W_2 (A - z I_{\cH})^{-1} W_1^*}   \no \\
& \qquad \times \big[(d/dz)W_2 (A - z I_{\cH})^{-1}\big]\Big),   \no \\
 & \quad = - \f{d}{dz} 
 \ln\big({\det}_{2,\cH}\big(I_{\cH} + \ol{W_2 (A - z I_{\cH})^{-1} W_1^*}\big)\big), 
\quad   z \in \rho(B) \cap \rho(A).     \lb{2.24} 
\end{align} 
Consequently, 
\begin{align}
& {\tr}_{\cH} \big((B - z I_{\cH})^{-1} - (A - z I_{\cH})^{-1}\big) 
= - \int_{\bbR} \f{\xi(\lambda; B, A) d\lambda}{(\lambda - z)^2} \no \\
& \quad = - \f{d}{dz}  \int_{\bbR} \xi(\lambda; B, A) d\lambda 
\bigg(\f{1}{\lambda - z} - \f{\lambda}{\lambda^2 + 1}\bigg)    \no \\
& \quad = - {\tr}_{\cH} \big(\ol{(A - z I_{\cH})^{-1} W (A - z I_{\cH})^{-1}}\big)\big)  
\no \\
& \qquad 
- \f{d}{dz} \ln\big({\det}_{2,\cH}\big(I_{\cH} + \ol{W_2 (A - z I_{\cH})^{-1} W_1^*}\big),   
\no \\
& \quad = - \eta'(z) 
- \f{d}{dz} \ln\big({\det}_{2,\cH}\big(I_{\cH} + \ol{W_2 (A - z I_{\cH})^{-1} W_1^*}\big), 
\quad   z \in \rho(B) \cap \rho(A),     \lb{2.25} 
\end{align} 
and hence,
\begin{align}
& \int_{\bbR} \xi(\lambda; B, A) d\lambda 
\bigg(\f{1}{\lambda - z} - \f{\lambda}{\lambda^2 + 1}\bigg)    \no \\
& \quad = \eta (z) 
+ \ln\big({\det}_{2,\cH}\big(I_{\cH} + \ol{W_2 (A - z I_{\cH})^{-1} W_1^*}\big)\big) + C, 
\quad z \in \rho(B) \cap \rho(A),    \lb{2.26} 
\end{align} 
for some $C\in\bbC$. Taking $z<0$, $|z|$ sufficiently large, \eqref{2.20} and 
\eqref{2.22} actually yield 
\begin{equation}
C \in\bbR.
\end{equation}
Moreover, \eqref{2.26} demonstrates that 
${\det}_{2,\cH}\big(I_{\cH} + \ol{W_2 (A - z I_{\cH})^{-1} W_1^*}\big)$ has normal limits 
$z \to \lambda + i0$ for a.e.\ $\lambda \in \bbR$. 
The Stieltjes inversion formula (cf., e.g., \cite{AD56}) then yields \eqref{2.21a}. 
\end{proof} 

We note that the analog of \eqref{2.21a} was discussed in \cite[Theorems 1.59 and 1.61]{Ko99} in the concrete context of multi-dimensional Schr\"odinger operators (an additional sign-definiteness of potentials was assumed for $n \geq 4$).

For subsequent purpose we record the analog of \eqref{2.11A} (cf.\ \eqref{2.25})
\begin{align}
& {\tr}_{\cH} \big((B - z I_{\cH})^{-1} - (A - z I_{\cH})^{-1}\big) 
= - \int_{\bbR} \f{\xi(\lambda; B, A) d\lambda}{(\lambda - z)^2} \no \\ 
& \quad = - {\tr}_{\cH} \big(\ol{(A - z I_{\cH})^{-1} W (A - z I_{\cH})^{-1}}\big)\big)  
\no \\
& \qquad 
- \f{d}{dz} \ln\big({\det}_{2,\cH}\big(I_{\cH} + \ol{W_2 (A - z I_{\cH})^{-1} W_1^*}\big),   
\quad   z \in \rho(B) \cap \rho(A).      \lb{2.28} 
\end{align} 

\medskip

For pertinent literature on (modified) Fredholm determinants and associated trace 
formulas we refer, for instance, to \cite{AS68}, \cite[Ch.\ IV]{GK69}, \cite[Sect.\ 1.6]{Ko99}, 
\cite{Si77}, \cite[Ch.\ 9]{Si05}, \cite[Sect.\ 1.7]{Ya92}, \cite{Ya07}, \cite[Chs.\ 3, 9]{Ya10}.

We conclude this preparatory section by analyzing the high-energy limiting assumptions \eqref{2.3a} and \eqref{2.20}. We start by recalling the following standard convergence property for trace ideals:

\begin{lemma} \lb{l2.4}
Let $p\in[1,\infty)$ and assume that $R,R_n,T,T_n\in\cB(\cH)$, 
$n\in\bbN$, satisfy
$\slim_{n\to\infty}R_n = R$  and $\slim_{n\to \infty}T_n = T$ and that
$S,S_n\in\cB_p(\cH)$, $n\in\bbN$, satisfy 
$\lim_{n\to\infty}\|S_n-S\|_{\cB_p(\cH)}=0$.
Then $\lim_{n\to\infty}\|R_n S_n T_n^\ast - R S T^\ast\|_{\cB_p(\cH)}=0$.
\end{lemma}
This follows, for instance, from \cite[Theorem 1]{Gr73}, \cite[p.\ 28--29]{Si05}, or
\cite[Lemma 6.1.3]{Ya92} with a minor additional effort (taking adjoints, etc.).

\begin{lemma} \lb{l2.5}
Let $p \in [1,\infty)$ and assume that $A$ is self-adjoint and bounded from below 
in $\cH$. Let $W_j$, $j=1,2$, be densely defined linear operators in $\cH$ satisfying
\begin{equation}
\dom(W_j) \supseteq \dom\big(|A|^{1/2}\big), \quad j =1,2,   \lb{2.29}
\end{equation}
and for some $z_0 \in \rho(A)$,
\begin{equation}
\ol{W_2 (A - z_0 I_{\cH})^{-1} W_1^*} \in \cB_p(\cH).    \lb{2.30} 
\end{equation}
In addition, suppose one of the following three conditions holds: 
\begin{align} 
& W_1 (A - {\ol z_0} I_{\cH})^{-1/2} \in \cB_q(\cH), 
\quad W_2 (A - z_0 I_{\cH})^{-1/2} \in \cB_r(\cH),     \lb{2.31} \\ 
& \hspace*{5.25cm}  \f{1}{q} + \f{1}{r} = \f{1}{p}, \; q, r \in [1,\infty),   \no \\     
& W_1 (A - {\ol z_0} I_{\cH})^{-1/2} \in \cB_p(\cH), 
\quad W_2 (A - z_0 I_{\cH})^{-1/2} \in \cB(\cH),    \lb{2.32} \\
& W_1 (A - {\ol z_0} I_{\cH})^{-1/2} \in \cB(\cH), 
\quad W_2 (A - z_0 I_{\cH})^{-1/2} \in \cB_p(\cH).   \lb{2.33}  
\end{align}
Then
\begin{equation}
\lim_{z \downarrow -\infty} \big\|\ol{W_2 (A - z I_{\cH})^{-1} W_1^*}\big\|_{\cB_p(\cH)} 
= \lim_{|\Im(z)|\to\infty} \big\|\ol{W_2 (A - z I_{\cH})^{-1} W_1^*}\big\|_{\cB_p(\cH)} = 0,  
\lb{2.34}
\end{equation}
and hence  
\begin{align}
\begin{split} 
& \lim_{z \downarrow -\infty} {\det}_{p,\cH}\big(I_{\cH} + \ol{W_2 (A - z I_{\cH})^{-1} W_1^*}\big)   \\
& \quad = \lim_{|\Im(z)|\to\infty}  {\det}_{p,\cH}\big(I_{\cH} + 
\ol{W_2 (A - z I_{\cH})^{-1} W_1^*}\big) = 1.     \lb{2.35}
\end{split} 
\end{align}
\end{lemma}
\begin{proof} We start with the identity
\begin{align}
& \ol{W_2 (A - z I_{\cH})^{-1} W_1^*} = \ol{W_2 \big[(A - z_0 I_{\cH})^{-1} + (z-z_0) 
(A - z_0 I_{\cH})^{-1} (A - z I_{\cH})^{-1}\big] W_1^*}    \no \\
& \quad = W_2 (A - z_0 I_{\cH})^{-1/2} \big[I_{\cH} + (z-z_0)(A - z I_{\cH})^{-1}\big]
\big[W_1 (A - \ol{z_0} I_{\cH})^{-1/2}\big]^*,     \lb{2.36} \\
& \hspace*{9.9cm} z \in \rho(A).     \no 
\end{align}
Combining \eqref{2.36} with the fact that 
\begin{align}
& \slim_{z\downarrow -\infty} \big[I_{\cH} + (z-z_0) (A - z I_{\cH})^{-1}\big] 
= \slim_{|\Im(z)| \to \infty} \big[I_{\cH} + (z-z_0) (A - z I_{\cH})^{-1}\big] = 0,  \\
& \slim_{z\downarrow -\infty} \big[I_{\cH} + (z-z_0) (A - z I_{\cH})^{-1}\big]^* 
= \slim_{|\Im(z)| \to \infty} \big[I_{\cH} + (z-z_0) (A - z I_{\cH})^{-1}\big]^* = 0,
\end{align}
then permits the application of Lemma \ref{l2.4} to conclude that 
\begin{equation}
\big\|\big[I_{\cH} + (z-z_0)(A - z I_{\cH})^{-1}\big]
\big[W_1 (A - \ol{z_0} I_{\cH})^{-1/2}\big]^*\big\|_{\cB_r(\cH)} \longrightarrow 0
\end{equation}
as $z\downarrow -\infty$ and also as $|\Im(z)|\to\infty$. This implies \eqref{2.32} in 
the case \eqref{2.31} is assumed. The cases where \eqref{2.32} or \eqref{2.33} are assumed are analogous. Continuity of ${\det}_{p,\cH}(I+T)$ as a function of $T$ with respect to the $\|\cdot\|_{\cB_p(\cH)}$-norm, $p \in [1,\infty)$, yields \eqref{2.33}.
\end{proof} 

The argument in the proof of Lemma \ref{l2.5} is analogous to  
the proof of \cite[Lemma\ 8.1.1]{Ya92} where the stronger relative trace class 
assumption $W_1^* W_2 (A - z_0)^{-1} \in \cB_1(\cH)$ is made.

\section{An Abstract Approach to Convergence of \\ Spectral Shift Functions}  \lb{s3}

In this section we prove our principal abstract result, the convergence of a sequence of spectral shift functions $\xi(\cdot;H_j,H_{0,j})$ associated with pairs of self-adjoint operators 
$(H_j, H_{0,j})$ in Hilbert spaces $\cH_j$, $j\in\bbN$, to the limiting 
spectral shift function $\xi(\cdot;H,H_0)$ associated with the pair $(H,H_0)$ in a limiting 
Hilbert space $\cH$ as $j \to\infty$ (mimicking the infinite volume limit in concrete 
situations).

We start with a precise list of our assumptions employed throughout this section:

\begin{hypothesis} \lb{h3.1} Let $\cH$ be a complex, separable Hilbert space. \\
$(i)$ Assume that $\{P_j\}_{j\in\bbN}$ is a sequence of orthogonal projections in $\cH$, strongly converging to the identity in $\cH$,
\begin{equation}
\slim_{j\to \infty} P_j = I_{\cH},     \lb{3.1} 
\end{equation}
and introduce the sequence of closed subspaces $\cH_j = P_j \cH$, $j\in\bbN$, of $\cH$. \\ 
$(ii)$ Let $H_0$ be a self-adjoint operator in $\cH$, and for each $j \in \bbN$, 
let $H_{0,j}$ be self-adjoint operators in $\cH_j$. In addition, suppose that 
$H_0$ is bounded from below in $\cH$, and for each $j\in\bbN$, $H_{0,j}$ are bounded 
from below in $\cH_j$. \\
$(iii)$ Suppose that $V_1$, and $V_2$ are closed operators in $\cH$, and for each 
$j \in\bbN$, assume that $V_{1,j}$, and $V_{2,j}$ are closed operators in $\cH_j$ such that 
\begin{align}
& \dom(V_1) \cap \dom(V_2) \supseteq \dom(H_0),    \lb{3.2} \\
& \dom(V_{1,j}) \cap \dom(V_{2,j}) \supseteq \dom(H_{0,j}), \quad j\in\bbN,   \lb{3.3}
\end{align} 
and  
\begin{align}
& \ol{V_2(H_0 - z I_{\cH})^{-1}V_1^*}, \, \ol{V_{2,j}(H_{0,j} - z I_{\cH_j})^{-1}V_{1,j}^*} 
\oplus 0 \in \cB_1(\cH), \quad j \in \bbN,    \lb{3.4} \\
& V_2(H_0 - z I_{\cH})^{-1}, \, V_{2,j}(H_{0,j} -z I_{\cH_j})^{-1}\oplus 0 \in \cB_{2}(\cH), 
\quad j\in\bbN,    \lb {3.5} \\
& \ol{(H_0 - z I_{\cH})^{-1}V_1^*}, \, \ol{(H_{0,j} - z I_{\cH_j})^{-1}V_{1,j}^*}\oplus 0 
\in \cB_{2}(\cH), \quad j\in\bbN,    \lb{3.6}
\end{align}
for some $($and hence for all\,$)$ $z\in \C\backslash \R$. In addition, assume that 
\begin{align}
\begin{split} 
& \lim_{z \downarrow - \infty} \big\|\big[\ol{V_2 (H_0 - z I_{\cH})^{-1}V_1^*}\big\|_{\cB_1(\cH)} =0, 
\lb{3.6a} \\ 
& \lim_{z \downarrow - \infty} \big\|\big[\ol{V_{2,j}(H_{0,j} - z I_{\cH_j})^{-1}V_{1,j}^*} 
\oplus 0\big]\big\|_{\cB_1(\cH)} =0, \quad j \in\bbN.  
\end{split}
\end{align}
Here we used the orthogonal decomposition of $\cH$ into
\begin{equation}
\cH=\cH_j\oplus \cH_j^{\perp}, \quad j\in\bbN.   \lb{3.7}
\end{equation}
$(iv)$ Assume that for some $($and hence for all\,$)$ $z\in \C\backslash \R$, 
\begin{equation}
\slim_{j \to \infty}\bigg[(H_{0,j} - z I_{\cH_j})^{-1}\oplus \frac{-1}{z} I_{\cH_j^{\perp}}\bigg] 
= (H_0 - z I_{\cH})^{-1}.    \lb{3.8}
\end{equation}
$(v)$ Suppose that for some $($and hence for all\,$)$ $z\in \C\backslash \R$, 
\begin{align}
& \lim_{j \to \infty}\big\|\big[\ol{V_{2,j}(H_{0,j} - z I_{\cH_j})^{-1}V_{1,j}^*}\oplus 0\big] 
- \ol{V_2 (H_0 - z I_{\cH})^{-1}V_1^*}\big\|_{\cB_1(\cH)} =0,    \lb{3.9} \\ 
& \lim_{j \to \infty} \big\|\big[V_{2,j}(H_{0,j} - z I_{\cH_j})^{-1}\oplus 0\big] - 
V_2 (H_{0} - z I_{\cH})^{-1}\big\|_{\cB_{2}(\cH)} =0,     \lb{3.10} \\ 
& \lim_{j \to \infty} \big\|\big[\ol{(H_{0,j} - z I_{\cH_j})^{-1}V_{1,j}^*}\oplus 0\big]
- \ol{(H_{0}-z)^{-1}V_1^*}\big\|_{\cB_{2}(\cH)} =0.    \lb{3.11}
\end{align}
$(vi)$ suppose that 
\begin{align}
\begin{split} 
& (V_2f,V_1g)_{\cH}=(V_1f,V_2g)_{\cH}, \quad f,g\in\dom(V_1)\cap\dom(V_2),   \lb{3.12} \\
& (V_{2,j}f,V_{1,j}g)_{\cH}=(V_{1,j}f,V_{2,j}g)_{\cH}, \quad f,g\in\dom(V_{1,j})
\cap\dom(V_{2,j}), \quad j \in\bbN. 
\end{split}
\end{align}
\end{hypothesis}

Following Kato \cite{Ka66}, Hypothesis \ref{h3.1}, permits one to define the 
self-adjoint operator $H$ in $\cH$, and for each $j\in\bbN$, the self-adjoint operators 
$H_j$ in $\cH_j$ via their resolvents (for $z\in\bbC\backslash\bbR$) by 
\begin{align}
& (H - z I_{\cH})^{-1} = (H_0 - z I_{\cH})^{-1}    \no \\ 
& \quad - \ol{(H_0 - z I_{\cH})^{-1}V_1^*} 
\big[I_{\cH} + \ol{V_2 (H_0 - z I_{\cH})^{-1} V_1^*}\big]^{-1} 
V_2(H_0 - z I_{\cH})^{-1},  \lb{3.13} \\
& (H_j - z I_{\cH_j})^{-1} = (H_{0,j} - z I_{\cH_j})^{-1}    \no \\ 
& \quad - \ol{(H_{0,j} - z I_{\cH_j})^{-1}V_{1,j}^*} 
\big[I_{\cH_j} + \ol{V_{2,j} (H_{0,j} - z I_{\cH_j})^{-1} V_{1,j}^*}\big]^{-1} 
V_{2,j}(H_0 - z I_{\cH_j})^{-1},   \no \\
& \hspace*{10cm}  j\in\bbN.   \lb{3.14}
\end{align}
Of course, both resolvent equations \eqref{3.13} and \eqref{3.14} extend by continuity 
to $\rho(H)\cap\rho(H_0)$ and $\rho(H_j)\cap\rho(H_{0,j})$, $j\in\bbN$, respectively. 

\begin{lemma} \lb{l3.2}
Assume Hypothesis \ref{h3.1}. Then
\begin{equation} 
\slim_{j \to \infty}\bigg[(H_j - z I_{\cH_j})^{-1}\oplus \frac{-1}{z} I_{\cH_j^{\perp}}\bigg] 
= (H - z I_{\cH})^{-1},  \quad z \in \bbC\backslash\bbR,  \lb{3.15} \\
\end{equation}
and
\begin{align}
& \big[(H - z I_{\cH})^{-1} - (H_0 - z I_{\cH})^{-1}\big] \in \cB_1(\cH),  \quad 
z \in \rho(H)\cap\rho(H_0),   \lb{3.16} \\
& \big[(H_j - z I_{\cH_j})^{-1} - (H_{0,j} - z I_{\cH_j})^{-1}\big] \in \cB_1(\cH_j), 
\quad z \in \rho(H_j)\cap\rho(H_{0,j}), \; j\in\bbN.  \lb{3.17} 
\end{align}
\end{lemma}
\begin{proof} 
To prove \eqref{3.15} it suffices to combine \eqref{3.8}--\eqref{3.11}, and \eqref{3.13} and \eqref{3.14}. Similarly, \eqref{3.16} and \eqref{3.17} follow upon combining 
\eqref{3.4}--\eqref{3.6} and \eqref{3.13} and \eqref{3.14}. 
 \end{proof}

One notes that the trace class and Hilbert-Schmidt assumptions in \eqref{3.4}--\eqref{3.6} and \eqref{3.9}--\eqref{3.11} are 
by no means necessary for the proof of \eqref{3.15}. In particular, $\cB_1(\cH)$ or $\cB_2(\cH)$ 
could be replaced by $\cB(\cH)$ in all these places (and we will use a Hilbert--Schmidt assumption later in the context of Theorem \ref{t3.17}).

\begin{remark} \lb{r3.1a}
We also note that if $H_{0,j} \geq c_0 I_{\cH_j}$ for some $c_0 \in\bbR$ independent 
of $j\in\bbN$, and if 
\begin{equation}
\lim_{z \downarrow - \infty} \big\|\big[\ol{V_{2,j}(H_{0,j} - z I_{\cH_j})^{-1}V_{1,j}^*} 
\oplus 0\big]\big\|_{\cB(\cH)} =0,  
\end{equation}
uniformly with respect to $j \in\bbN$, then \eqref{3.14} shows that also $H_{j}$ are bounded from below, uniformly with respect to $j\in\bbN$, that is, $H_{j} \geq c I_{\cH_j}$ for some $c \in\bbR$ independent of $j\in\bbN$. 
\end{remark}

Assuming Hypothesis \ref{h3.1}, we now abbreviate by 
\begin{equation}
\xi(\cdot) = \xi(\cdot;H,H_0), \quad \xi_j(\cdot) = \xi(\cdot;H_j,H_{0,j}), \quad j\in\bbN, 
\lb{3.18}
\end{equation}
the Krein spectral shift functions corresponding to the pairs $(H,H_0)$ in $\cH$ and 
$(H_j,H_{0,j})$ in $\cH_j$,  respectively. Thus, for $z\in \C\backslash \R$,
\begin{align}
\begin{split} 
{\tr}_{\cH} \big((H - z I_{\cH})^{-1} - (H_{0} - z I_{\cH})^{-1}\big) 
& = -\int_{\R}\frac{\xi(\lambda) d\lambda}{(\lambda-z)^2},    \\
{\tr}_{\cH_j} \big((H_j - z I_{\cH_j})^{-1} - (H_{0,j} - z I_{\cH_j})^{-1}\big) 
& = - \int_{\R}\frac{\xi_j (\lambda) d\lambda}{(\lambda-z)^2},    \lb{3.20} 
\end{split} 
\end{align}
with, 
\begin{equation}\lb{3.21}
\xi , \, \xi_j\in L^1\big(\R; (1+\lambda^2)^{-1}d\lambda\big), \quad j\in\bbN. 
\end{equation}
In addition, we introduce the perturbation determinants
\begin{align}
D(z)& = {\det}_{\cH}\big(I_{\cH}+\ol{V_2 (H_{0} - z I_{\cH})^{-1}V_1^*}\big), 
\quad z\in \rho(H) \cap \rho(H_0),   \no \\
D_j(z)& = {\det}_{\cH_j}\big(I_{\cH_j} + \ol{V_{2,j}(H_{0,j} - z I_{\cH_j})^{-1}V_{1,j}^*}\big), 
\quad z\in \rho(H_j) \cap \rho(H_{0,j}), \; j\in\bbN.   \lb{3.23}  
\end{align}

We start with the following preliminary results.

\begin{lemma} \lb{l3.3}
Assume Hypothesis \ref{h3.1} and let $a, z \in \bbC\backslash \bbR$. Then
\begin{align}
& \lim_{j \to \infty}\ln \big(D_j(z)/D_j(a)\big)=\ln \big(D(z)/D(a)\big),  
\lb{3.24} \\
& \lim_{j \to \infty}\int_{\R}\frac{\xi_j(\lambda) d\lambda}{(\lambda -a )(\lambda-z)^n}
=\int_{\R}\frac{\xi(\lambda) d\lambda}{(\lambda-a)(\lambda-z)^n}, \quad n\in \N.  
\lb{3.25}
\end{align}
\end{lemma}
\begin{proof}
The identities
\begin{align}
& {\det}_{\cH_j}\big(I_{\cH_j}+\ol{V_{2,j}(H_{0,j} - z I_{\cH_j})^{-1}V_{1,j}^*} \big)   \no \\
& \quad = {\det}_{\cH}\big(\big(I_{\cH_j}+\ol{V_{2,j}(H_{0,j} - z I_{\cH_j})^{-1}V_{1,j}^*}
\big)\oplus 
I_{\cH_j^{\perp}} \big)\no\\
& \quad = {\det}_{\cH}\big(I_\cH + \big(\ol{V_{2,j}(H_{0,j} - z I_{\cH_j})^{-1}V_{1,j}^*}
\oplus 0\big)\big),  \quad j\in\bbN, \lb{3.26}
\end{align}
together with \eqref{3.9} and continuity of ${\det}_{\cH}(I+A)$ as a function of $A$ with respect to the trace norm $\|\cdot\|_{\cB_1(\cH)}$, immediately yield
\begin{equation}
\lim_{j \to \infty}D_j(z)=D(z), \quad z\in \C\backslash \R,     \lb{3.27}
\end{equation}
and hence the convergence in \eqref{3.24}.

Applying \eqref{2.7}, one verifies that for any $a, z \in \bbC\backslash \bbR$, 
\begin{align}
\ln\big(D(z)/D(a)\big)&=(z-a)\int_{\R}\frac{\xi(\lambda) d\lambda}
{(\lambda-a)(\lambda-z)},     \lb{3.28} \\
\ln\big(D_j(z)/D_j(a)\big)&=(z-a)\int_{\R}\frac{\xi_j(\lambda) d\lambda}
{(\lambda-a)(\lambda-z)}, \quad j \in \bbN.     \lb{3.29}
\end{align}

To verify \eqref{3.25}, we start with the basic identities (see, e.g., \cite[Ch.\ 8]{Ya92}) 
\begin{align}
\tr_{\cH}\big((H_{0}-zI_{\cH})^{-n}-(H-zI_{\cH})^{-n}\big) 
&=n\int_{\R}\frac{\xi(\lambda) d\lambda}{(\lambda-z)^{n+1}},    \lb{3.30} \\
\tr_{\cH_j}\big ((H_{0,j}-zI_{\cH_j})^{-n}-(H_j-zI_{\cH_j})^{-n}\big) 
&=n\int_{\R}\frac{\xi_j(\lambda) d\lambda}{(\lambda-z)^{n+1}}, \quad  
j \in \bbN,   \lb{3.31} \\
& \hspace*{1.6cm} n\in \N, \quad z\in \C\backslash \R.\no
\end{align}
Next we claim that 
\begin{align}
\begin{split} 
& \lim_{j \to \infty} \tr_{\cH_j}\big((H_{0,j}-zI_{\cH_j})^{-n}-(H_j-zI_{\cH_j})^{-n} \big)   \\
& \quad =\tr_{\cH}\big((H_{0}-zI_{\cH})^{-n}-(H-zI_{\cH})^{-n}\big),   \quad 
n\in \N, \quad z\in \C\backslash \R.    \lb{3.32}
\end{split} 
\end{align}
To see this, one notes that
\begin{align}
& \tr_{\cH_j}\big((H_{0,j}-zI_{\cH_j})^{-n}-(H_j-zI_{\cH_j})^{-n} \big)    \no \\
& \quad = \tr_{\cH}\big(((H_{0,j}-zI_{\cH_j})^{-n}-(H_j-zI_{\cH_j})^{-n})\oplus 0 \big)\no\\
 &\quad = \tr_{\cH}\bigg[\bigg((H_{0,j}-zI_{\cH_j})^{-1}
 \oplus \frac{-1}{z} I_{\cH_j^\perp}\bigg)^n 
 - \bigg((H_j-zI_{\cH_j})^{-1}\oplus \frac{-1}{z} I_{\cH_j^\perp}\bigg)^n  \bigg],   \no \\
& \hspace*{8.3cm} n\in \N, \; z\in \C\backslash \R.    \lb{3.33}
\end{align}
Since the trace functional is continuous with respect to the 
$\cB_1(\cH)$-norm, to verify \eqref{3.32}, it suffices to prove that 
\begin{align}   \lb{3.34}
& \sum_{k=1}^n \bigg((H_{0,j}-zI_{\cH_j})^{-1}\oplus\frac{-1}{z} I_{\cH_j^\perp}\bigg)^{n-k}
\bigg[(H_{0,j}-zI_{\cH_j})^{-1}-(H_j-zI_{\cH_j})^{-1}\oplus 0 \bigg]    \no \\
& \quad \times \bigg((H_j-zI_{\cH_j})^{-1} 
\oplus \frac{-1}{z} I_{\cH_j^\perp}\bigg)^{k-1}
\end{align}
converges to 
\begin{equation}\lb{3.35}
\sum_{k=1}^n (H_0-zI_{\cH})^{k-n}\big[(H_0-zI_{\cH})^{-1}
-(H-zI_{\cH})^{-1} \big](H-zI_{\cH})^{1-k}
\end{equation}
in $\cB_1(\cH)$ as $j\to \infty$, since \eqref{3.34} is the operator under the trace on 
the r.h.s. of \eqref{3.33} and \eqref{3.35} is the operator under the trace on the 
r.h.s. of \eqref{3.32}.\footnote{Here we have made use of the identity 
$A^n-B^n=\sum_{k=1}^nA^{n-k}(A-B)B^{k-1}$.}

By \eqref{3.8}, one concludes that 
\begin{align}
\slim_{j \to \infty} \bigg((H_{0,j}-zI_{\cH_j})^{-1}\oplus\frac{-1}{z} I_{\cH_j^\perp}\bigg)^{n-k}
&=(H_{0}-zI_{\cH})^{k-n},  
\lb{3.36} \\
\slim_{j \to \infty} \bigg((H_j-zI_{\cH_j})^{-1}\oplus\frac{-1}{z} I_{\cH_j^\perp}\bigg)^{k-1}
&=(H-zI_{\cH})^{1-k}.  \lb{3.37}
\end{align}
Thus, convergence of \eqref{3.34} to \eqref{3.35}, will follow from Gr\"{u}mm's 
Theorem \cite{Gr73} (see also Lemma \ref{l2.4} and the discussion in \cite[Ch.\ 2]{Si05}) 
if we can show that \begin{align} \lb{3.38}
\begin{split} 
& \lim_{j \to \infty} \big\|\big[\big((H_{0,j}-zI_{\cH_j})^{-1}-(H_j-zI_{\cH_j})^{-1}\big)
\oplus 0\big]     \\
& \hspace*{1cm} - (H_0-zI_{\cH})^{-1}-(H-zI_{\cH})^{-1}\big\|_{\cB_1(\cH)} = 0. 
\end{split} 
\end{align}
 The convergence in \eqref{3.38} follows readily from the identities 
\begin{align}
&\big((H_j-zI_{\cH_j})^{-1}-(H_{0,j}-zI_{\cH_j})^{-1}\big)\oplus 0
=\big(\ol{(H_{0,j}-zI_{\cH_j})^{-1}V_{1,j}^*}\oplus0\big)  \lb{3.39} \\
&\quad \times 
\Big(\big[I_{\cH}+\big(\ol{V_{2,j}(H_{0,j}-zI_{\cH_j})^{-1}V_{1,j}^*}\oplus 0\big) \big]^{-1}
-\big(0\oplus I_{\cH_j^{\perp}} \big)\Big)    \no \\
& \quad \times \big(V_{2,j}(H_{0,j}-zI_{\cH_j})^{-1}
\oplus 0\big),     \no \\
&(H_{0}-zI_{\cH})^{-1}-(H-zI_{\cH})^{-1}    \no\\
&\quad = \ol{(H_{0}-zI_{\cH})^{-1}V_1^*}\big[I_{\cH}
+\ol{V_2(H_{0}-zI_{\cH})^{-1}V_1^*} \big]^{-1}
V_2(H_{0}-zI_{\cH})^{-1}.\lb{3.40}
\end{align}
Hypothesis \eqref{3.1} and relation \eqref{3.9} yield the strong convergence
\begin{align}\lb{3.41}
\begin{split} 
& \slim_{j\to \infty}\big(\big[I_{\cH}+\big(\ol{V_{2,j}(H_{0,j}-zI_{\cH_j})^{-1}V_{1,j}^*}
\oplus 0\big) \big]^{-1} - \big(0 \oplus I_{\cH_j^{\perp}} \big)\big)    \\
& \quad = \big[I+\ol{V_2(H_{0}-zI_{\cH})^{-1}V_1^*}\big]^{-1}.
\end{split} 
\end{align}
Therefore, \eqref{3.41} and \eqref{3.10} together with Gr\"{u}mm's Theorem \cite{Gr73} 
yield 
\begin{align}
&\lim_{j\to \infty}\Big(\big[I_{\cH}+\big(\ol{V_{2,j}(H_{0,j}-zI_{\cH_j})^{-1}V_{1,j}^*}
\oplus 0\big) \big]^{-1}   \no \\ 
& \qquad -\big(0 \oplus I_{\cH_j^{\perp}} \big)\Big)\big(V_{2,j}(H_{0,j}-zI_{\cH_j})^{-1}
\oplus 0\big)     \no \\
&\quad =\Big[I+\ol{V_2(H_{0}-zI_{\cH})^{-1}V_1^*}\Big]^{-1} 
V_2 (H_0-z I_{\cH})^{-1} \, \text{ in $\cB_2(\cH)$.}   \lb{3.42}
\end{align}
The convergence in \eqref{3.42} and \eqref{3.11}  
yields convergence of the r.h.s. of \eqref{3.39} to \eqref{3.40} in $\cB_1(\cH)$, implying \eqref{3.38}.

Employing \eqref{3.31}, \eqref{3.30}, and \eqref{3.38}, we have shown that
\begin{equation}\lb{3.43}
\lim_{j\to \infty}\int_{\R}\frac{\xi_j(\lambda) d\lambda}{(\lambda-z)^{n+1}} 
=\int_{\R}\frac{\xi(\lambda) d\lambda}{(\lambda-z)^{n+1}}, \quad n\in \N
\end{equation}
so that \eqref{3.25} holds in the special case $a=z$. Thus, it remains to settle 
the case $z\neq a$.

For $z\neq a$ one notes that 
\begin{align}
\int_{\R}\frac{\xi_j(\lambda) d\lambda}{(\lambda-a)(\lambda-z)^{n+1}}
&=\frac{1}{z-a}\bigg[\int_{\R}\frac{\xi_j(\lambda) d\lambda}{(\lambda - z)^{n+1}} 
- \int_{\R}\frac{\xi_j(\lambda) d\lambda}{(\lambda-a)(\lambda-z)^n}\bigg],   \no \\
\int_{\R}\frac{\xi(\lambda) d\lambda}{(\lambda-a)(\lambda-z)^{n+1}} 
&=\frac{1}{z-a}\bigg[\int_{\R}\frac{\xi(\lambda) d\lambda}{(\lambda - z)^{n+1}} 
-\int_{\R}\frac{\xi(\lambda) d\lambda}{(\lambda-a)(\lambda-z)^n}\bigg],    \lb{3.44} \\
& \hspace*{6.4cm}   n \in \bbN.  \no 
\end{align}
Convergence in \eqref{3.25} now follows from \eqref{3.43} and the two identities 
in \eqref{3.44} via a simple induction on $n$.  We emphasize that \eqref{3.24} 
yields the crucial first induction step, $n=1$, since \eqref{3.24} implies, 
via \eqref{3.29} and \eqref{3.28}, that
\begin{equation}
\lim_{j\to \infty}\int_{\R}\frac{\xi_j(\lambda) d\lambda}{(\lambda-a)(\lambda-z)} 
= \int_{\R}\frac{\xi(\lambda) d\lambda}{(\lambda-a)(\lambda-z)}, 
\quad z\neq a.  \lb{3.45}
\end{equation}
\end{proof}

In the following we denote by $C_{\infty}(\R)$ the space of continuous functions on $\bbR$ vanishing at infinity.

\begin{lemma} \lb{l3.4}
Let $f, f_j\in L^1(\bbR; d\lambda)$ and suppose that for some fixed $M>0$, 
$\|f_j\|_{L^1(\bbR;d\lambda)}\leq M$, $j\in\bbN$.  If
\begin{equation} \lb{3.46}
\lim_{j\rightarrow \infty} \int_{\bbR} f_j(\lambda) d\lambda \, 
P((\lambda+i)^{-1},(\lambda-i)^{-1}) = \int_{\bbR}  
f(\lambda) d\lambda \, P((\lambda+i)^{-1},(\lambda-i)^{-1})
\end{equation}
for all polynomials $P(\cdot,\cdot)$ in two variables, then
\begin{equation}
\lim_{j\rightarrow \infty}\int_{\bbR} 
f_j(\lambda) d\lambda \, g(\lambda)=\int_{\bbR}  
f(\lambda) d\lambda \,g(\lambda), \quad g\in C_{\infty}(\bbR). 
\end{equation}
\end{lemma}
\begin{proof}
Let $\varepsilon>0$ and $g\in C_{\infty}(\bbR)$.  Since by a 
Stone--Weierstrass argument, polynomials in 
$(\lambda \pm i)^{-1}$ are dense in $C_{\infty}(\bbR)$, there is a polynomial $P(\cdot,\cdot)$ in two variables such that writing 
\begin{equation}
\cP(\lambda)=P((\lambda+i)^{-1},(\lambda-i)^{-1}), 
\quad \lambda \in\bbR,
\end{equation}
one concludes that  
\begin{equation}
\|g-\cP\|_{L^\infty(\bbR; d\lambda)} \leq 
\frac{\varepsilon}{2[M+\|f\|_{L^1(\bbR;d\lambda)}]}. 
\end{equation}
By \eqref{3.46}, there exists an $N(\varepsilon) \in \bbN$ such that
\begin{equation}
\bigg|\int_{\bbR} f_j(\lambda) d\lambda \, \cP(\lambda) -\int_{\bbR} 
f(\lambda) d\lambda \, \cP(\lambda) \bigg| \leq \frac{\varepsilon}{2}
\, \text{ for all $j \geq N(\varepsilon)$.} 
\end{equation}
Therefore, if $j \geq N(\varepsilon)$, 
\begin{align}
& \bigg| \int_{\bbR} f_j(\lambda)d\lambda \, g(\lambda)  
- \int_{\bbR} f(\lambda) d\lambda \, g(\lambda) \bigg| 
\leq \big[\|f_j\|_{L^1(\bbR;d\lambda)}+\|f\|_{L^1(\bbR;d\lambda)}\big] 
\|g-\cP\|_{L^\infty(\bbR; d\lambda)}   \no \\
& \quad + \bigg|\int_{\bbR} f_j(\lambda) d\lambda \, \cP(\lambda) 
-\int_{\bbR} f(\lambda) d\lambda \, \cP(\lambda)\bigg| 
\leq \varepsilon.  
\end{align}
\end{proof}

Next, we continue with some preparations needed to prove the principal results of this section. We start by recalling some basic notions regarding the convergence of positive measures (essentially following Bauer 
\cite[\&\ 30]{Ba01}). Denoting by $\mathscr{M}_+(E)$ the set of all {\it positive Radon measures} on a locally compact space $E$, and by 
\begin{equation}
\mathscr{M}_+^b(E)=\{\mu \in \mathscr{M}_+(E)\,|\, \mu(E)<+\infty\}, 
\end{equation}
the set of all \textit{finite} positive Radon measures on $E$, we note that in the special case $E=\R^n$, $n\in\bbN$, $\mathscr{M}_+^b(\R^n)$ represents the set of all finite positive Borel measures on $\R^n$.  

If $\mu$ is a Radon measure, a point $x\in E$ is called an \textit{atom} of $\mu$ if $\mu(\{x\})>0$.  

In the following, $C_0(E)$ denotes the continuous functions on $E$ with compact support, and $C_b(E)$ represents the bounded continuous functions on $E$.

\begin{definition} \lb{d3.5}
Let $E$ be a locally compact space. \\
$(i)$ A sequence $\{\mu_j\}_{j\in\bbN} \subset \mathscr{M}_+(E)$ is said to be \textit{vaguely convergent} to a Radon measure $\mu \in \mathscr{M}_+(E)$ if
\begin{equation}
\lim_{j\to \infty}\int_E d\mu_j \, g = \int_E d\mu \, g, \quad g \in C_0(E). 
\end{equation}
$(ii)$ A sequence $\{\mu_j\}_{j\in\bbN} \subset \mathscr{M}_+^b(E)$ is said to be \textit{weakly convergent} to $\mu \in \mathscr{M}_+^b(E)$ if 
\begin{equation}
\lim_{j \to \infty}\int_E d\mu_j \, f=\int_Ed\mu \, f, \quad f\in C_b(E). 
\end{equation}
$(iii)$ A Borel set $B \subset E$ is called \textit{boundaryless} with respect to the measure $\mu \in \mathscr{M}_+^b(E)$  
$($in short, $\mu$-\textit{boundaryless}$)$, if the boundary $\partial B$ of $B$ has $\mu$-measure equal to zero, $\mu(\partial B) = 0$.
\end{definition} 

\begin{theorem}[\cite{Ba01}, Theorem 30.8] \lb{t3.7}
Suppose that the sequence $\{\mu_j\}_{j\in\bbN} \subset \mathscr{M}_+^b(E)$ converges vaguely to the measure $\mu \in \mathscr{M}_+^b(E)$.  Then the following statements are equivalent: \\
$(i)$ The sequence $\mu_j$ converges weakly to $\mu$ as $j\to\infty$.  \\
$(ii)$ $\lim_{j \rightarrow \infty} \mu_j (E) = \mu (E)$.   \\
$(iii)$ For every $\varepsilon>0$ there exists a compact set $K_{\varepsilon}$  of $E$ such that 
\begin{equation}
\mu_j (E \backslash K_{\varepsilon}) \leq \varepsilon, \quad j \in \N.
\end{equation}
\end{theorem}

\begin{theorem}[\cite{Ba01}, Theorem 30.12]\lb{t3.8}
Suppose that the sequence $\{\mu_j\}_{j\in\bbN} \subset \mathscr{M}_+^b(E)$ converges weakly to $\mu \in \mathscr{M}_+^b(E)$. Then
\begin{equation}
\lim_{j \rightarrow \infty}\int_E d\mu_j \, f = \int_E d\mu \, f
\end{equation}
holds for every bounded Borel measurable function $f$ that is $\mu$-almost everywhere continuous on $E$.  In particular, 
\begin{equation}
\lim_{j \rightarrow \infty}\mu_j (B)=\mu(B)
\end{equation}
holds for every $\mu$-boundaryless Borel set $B$.
\end{theorem}

As usual, finite signed Radon measures are viewed as differences of finite positive Radon measures in the following. 

Next, we slightly strengthen our assumptions a bit. 

\begin{hypothesis} \lb{h3.9}
In addition to Hypothesis \ref{h3.1} we now assume the following conditions: \\
$(vii)$ Suppose that $V_1$, and $V_2$ are closed operators in $\cH$, and for each 
$j \in\bbN$, assume that $V_{1,j}$, and $V_{2,j}$ are closed operators in $\cH_j$ such that 
\begin{align}
& \dom(V_1) \cap \dom(V_2) \supseteq \dom\big(|H_0|^{1/2}\big),    \lb{3.60} \\
& \dom(V_{1,j}) \cap \dom(V_{2,j}) \supseteq \dom\big(|H_{0,j}|^{1/2}\big), \quad j\in\bbN,   
\lb{3.61}
\end{align} 
and  
\begin{equation}
V = V_1^* V_2 \, \text{ is a self-adjoint operator in $\cH$},    \lb{3.62}
\end{equation}
and for each $j\in\bbN$,
\begin{equation}
V_j = V_{1,j}^* V_{2,j} \, \text{ is a self-adjoint operator in $\cH_j$}.    \lb{3.63}
\end{equation}
$(viii)$ Decomposing $V, V_j$, $j\in\bbN$, into their positive and negative parts, 
\begin{equation}
V_{\pm} = (1/2) [|V| \pm V], \quad V_{j,\pm} = (1/2) [|V_j| \pm V_j], \; j\in\bbN,  
\lb{3.64}
\end{equation}
$V_{\pm}$ are assumed to be infinitesimally form bounded with respect to $H_0$, 
and for each $j\in\bbN$, $V_{j,\pm}$ are assumed to be infinitesimally form bounded 
with respect to $H_{0,j}$.
\end{hypothesis}

Hypothesis \ref{h3.9} permits us to identify $H$ and $H_j$ with the form sums,
\begin{equation}
H = H_0 +_q V, \quad H_j = H_{0,j} +_q V_j, \; j\in\bbN.    \lb{3.65}
\end{equation}
It also permits one to introduce the positive and negative parts of $V$ and $V_j$ step by step, and in either order, that is,
\begin{align}
H &= (H_0 +_q V_+) +_q (-V_-) = H_0 +_q V_+ -_q V_-,   \lb{3.66} \\ 
H_j &= (H_{0,j} +_q V_{j,+}) +_q (-V_{j,-}) 
= H_{0,j} +_q V_{j,+} -_q V_{j,-}, \quad j\in\bbN,   \lb{3.67}
\end{align}
with resolvent equations of the type \eqref{3.13} and \eqref{3.14} valid in each case 
(replacing $H_0$, $H_{0,j}$ by $H_0 +_q V_+$, $H_{0,j} +_q V_{j,+}$, etc.). 

In this context we now decompose
\begin{align}
\xi(\cdot) &= \xi(\cdot;H,H_0) = \xi(\cdot;H_0 +_q V_+ -_q V_- ,H_0)   \no \\
& = \xi(\cdot;H_0 +_q V_+ -_q V_- ,H_0 +_q V_+) 
+ \xi(\cdot;H_0 +_q V_+,H_0),    \no \\
& = \xi_+(\cdot) - \xi_-(\cdot),    \lb{3.68} \\
\xi_j(\cdot) &= \xi(\cdot;H_j,H_{0,j}) 
= \xi(\cdot;H_{0,j} +_q V_{j,+} -_q V_{j,-} ,H_{0,j})   \no \\
& = \xi(\cdot;H_{0,j} +_q V_{j,+} -_q V_{j,-} ,H_{0,j} +_q V_{j,+}) 
+ \xi(\cdot;H_{0,j} +_q V_{+,j},H_{0,j})     \no \\
& = \xi_{j,+}(\cdot) - \xi_{j,-}(\cdot),    \quad j \in \bbN,    \lb{3.69}
\end{align}
where
\begin{align}
& \xi_+(\cdot) = \xi(\cdot;H_0 +_q V_+,H_0) \geq 0,  \lb{3.70} \\
& \xi_-(\cdot) = - \xi(\cdot;H_0 +_q V_+ -_q V_- ,H_0 +_q V_+) \geq 0, 
\lb{3.71} \\
& \xi_{j,+}(\cdot) = \xi(\cdot;H_{0,j} +_q V_{+,j},H_{0,j}) \geq 0, \quad j\in\bbN, 
\lb{3.72} \\
& \xi_{j,-}(\cdot) 
= - \xi(\cdot;H_{0,j} +_q V_{j,+} -_q V_{j,-} ,H_{0,j} +_q V_{j,+}) \geq 0, 
\quad j \in \bbN.   \lb{3.73}
\end{align}

\begin{theorem}\lb{t3.10}
Assume Hypothesis \ref{h3.9} and $g \in C_\infty(\bbR)$. Then 
\begin{equation}\lb{3.74}
\lim_{j\to \infty} \int_{\R} \frac{\xi_j(\lambda) d\lambda}{\lambda^2 + 1} \, g(\lambda) 
= \int_{\R} \frac{\xi(\lambda) d\lambda}{\lambda^2 + 1} \, g(\lambda). 
\end{equation} 
\end{theorem}
\begin{proof}
The basic idea of proof  consists of verifying that
\begin{equation}\lb{3.75}
\lim_{j\to \infty}\int_{\R}\frac{\xi_{j,\pm}(\lambda) d\lambda}{\lambda^2 + 1} \, 
P((\lambda+i)^{-1},(\lambda-i)^{-1})
=\int_{\R}\frac{\xi_{\pm}(\lambda) d\lambda}{\lambda^2 + 1} \, 
P((\lambda+i)^{-1},(\lambda-i)^{-1} )
\end{equation}
for all polynomials $P(\cdot,\cdot)$ in two variables, and then rely on the 
Stone--Weierstrass approximation in Lemma \ref{l3.4} to get  
\begin{equation} 
 \lim_{j\rightarrow \infty}\int_{\R}\frac{\xi_{j,\pm}(\lambda) d\lambda}{\lambda^2 + 1} \, g(\lambda)   
= \int_{\R}\frac{\xi_{\pm}(\lambda) d\lambda}{\lambda^2 + 1} \, g(\lambda), 
\quad g\in C_{\infty}(\R),    \lb{3.76}
\end{equation}
and hence \eqref{3.74}. To prove \eqref{3.75}, it suffices to verify
\begin{align}\lb{3.77}
\begin{split} 
\lim_{j\to \infty}\int_{\R}\frac{\xi_{j,\pm}(\lambda) d\lambda}{\lambda^2 + 1} 
\frac{1}{(\lambda+i)^m(\lambda-i)^n}
= \int_{\R}\frac{\xi_{\pm}(\lambda) d\lambda}{\lambda^2 + 1}
\frac{1}{(\lambda+i)^m(\lambda-i)^n},& \\   
m,n\in \N\cup \{0\},& 
\end{split} 
\end{align}
which, in turn, follows once one proves 
\begin{equation}\lb{3.78}
\lim_{j\to \infty}\int_{\R}\frac{\xi_{j,\pm}(\lambda) d\lambda}{\lambda^2 + 1} 
\frac{1}{(\lambda\pm i)^n}
= \int_{\R}\frac{\xi_{\pm}(\lambda) d\lambda}{\lambda^2 + 1} 
\frac{1}{(\lambda\pm i)^n}, \quad  n\in \N\cup \{0\}, 
\end{equation}
since
\begin{equation}
\frac{1}{(\lambda+i)^m(\lambda-i)^n}
=\sum_{j=1}^m\frac{c_j}{(\lambda+i)^j}+\sum_{j=1}^n\frac{\widehat{c}_j}{(\lambda-i)^j}
\end{equation}
for appropriate sets of constants $c_j$ and $\widehat{c}_j$.
Choosing $z=\pm i$ and $a=\mp i$ in \eqref{3.25} yields \eqref{3.78}, and therefore 
\eqref{3.75} for all polynomials $P$.  At this point, \eqref{3.74} follows from 
Lemma \ref{l3.4} once one shows the existence of an $M>0$ for which 
\begin{equation}
\int_{\R} \f{\xi_{j,\pm}(\lambda) d\lambda)}{\lambda^2 +1} \leq M 
\end{equation} 
for $j$ sufficiently large. Taking \eqref{3.78} with $n=0$, yields the convergence  
\begin{equation} 
\lim_{j\to\infty} \int_{\R} \f{\xi_{j,\pm}(\lambda) d\lambda}{\lambda^2 +1} 
= \int_{\R} \f{\xi_{\pm}(\lambda) d\lambda}{\lambda^2 +1}.  
\end{equation} 
As a result, $\int_{\R} \xi_{j,\pm}(\lambda) d\lambda \, (1+\lambda^2)^{-1}$ is uniformly bounded with respect to $j \in \bbN$.
\end{proof}

An immediate consequence of Theorem \ref{t3.10} is the following vague convergence result:

\begin{corollary}\lb{c3.11}
Assume Hypothesis \ref{h3.9} and let $g \in C_0(\bbR)$. Then 
\begin{equation}\lb{3.82}
\lim_{j\to \infty} \int_{\R}\xi_j(\lambda)d\lambda \, g(\lambda) 
= \int_{\R}\xi(\lambda) d\lambda \, g(\lambda). 
\end{equation} 
\end{corollary}

\begin{remark} \lb{r3.12}
If the operators $H_{0,j}$ and $H_j$ are actually uniformly bounded from below 
with respect to $j\in\bbN$, and hence according to our convention \eqref{2.9}, 
$\xi$ and $\xi_j$, $j\in\bbN$, are chosen to be zero in a fixed (i.e., $j$-independent) neighborhood of $-\infty$, then no condition need be imposed on $g$ in a neighborhood 
of $-\infty$ in \eqref{3.74} and \eqref{3.82} (apart from measurability of $g$, of course).
\end{remark}

Given the decomposition \eqref{3.68}--\eqref{3.73}, and introducing the measures
\begin{align}
\begin{split}
\eta_{\xi_{\pm}} (A)=\int_A\frac{\xi_{\pm} 
(\lambda) d\lambda}{\lambda^2 + 1},
\quad \eta_{\xi_{j,\pm}} (A)=\int_A\frac{\xi_{j,\pm} 
(\lambda) d\lambda}{\lambda^2 + 1}, \; j\in\bbN,&    \lb{3.83} \\ 
\text{ $A \subseteq \bbR$ Lebesgue measurable,}&
\end{split}
\end{align}
we are now ready for the principal result of this paper.

\begin{theorem}\lb{t3.13}
Assume Hypothesis \ref{h3.9}. Then
\begin{equation}\lb{3.84}
\lim_{j\rightarrow \infty}\int_{\R} 
\frac{\xi(\lambda; H_j,H_{0,j}) d\lambda}{\lambda^2 + 1} \, f(\lambda) 
= \int_{\R} \frac{\xi(\lambda;H,H_0) d\lambda}{\lambda^2 + 1} \, f(\lambda),  
\quad f\in C_b(\R).
\end{equation} 
Moreover, \eqref{3.84} holds for every bounded Borel measurable function $f$ that is $\eta_{\xi_+}$ and $\eta_{\xi_-}$-almost everywhere continuous 
on $\R$.  In particular,
\begin{equation}\lb{3.85}
\lim_{j\rightarrow \infty} \int_S 
\frac{\xi(\lambda; H_j,H_{0,j}) d\lambda}{\lambda^2 + 1} 
= \int_S \frac{\xi(\lambda;H,H_0) d\lambda}{\lambda^2 + 1}
\end{equation}
for every set $S$ whose boundary has $\eta_{\xi_+}$ and 
$\eta_{\xi_-}$-measure equal to zero.
\end{theorem}
\begin{proof}
Again we decompose 
$\xi$ and $\xi_j$, $j\in\bbN$, as in \eqref{3.68}--\eqref{3.73}. By \eqref{3.76} and 
Corollary \ref{c3.11}, the measure $\eta_{\xi_{j,\pm}}$ vaguely converges to the measure 
$\eta_{\xi_{\pm}}$ as $j\to\infty$, respectively. Moreover, by \eqref{3.77},
\begin{equation}\label{3.86a}
\lim_{j\to\infty} \eta_{\xi_{j,\pm}} (\R) = \eta_{\xi_{\pm}} (\R). 
\end{equation}
Thus, by Theorem \ref{t3.7}, one concludes weak convergence of the sequence of measures $\eta_{\xi_{j,\pm}}$ to the measure $\eta_{\xi_{\pm}}$ as $j\to\infty$.  That 
\eqref{3.84} holds for every bounded Borel measurable function that is 
$\eta_{\xi_{\pm}}$-almost everywhere continuous on $\R$ now follows directly from Theorem \ref{t3.8}. Finally, convergence in \eqref{3.85} is also a direct consequence of 
Theorem \ref{t3.8}.
\end{proof}

As immediate consequences of Theorem \ref{t3.13}, we have the following two results:

\begin{corollary} \lb{c3.14}
Assume Hypothesis \ref{h3.9}. Then convergence in \eqref{3.84} holds for any bounded Borel measurable function that is Lebesgue-almost everywhere continuous.  In particular, 
\eqref{3.85} holds for any set $S$ that is boundaryless with respect to Lebesgue 
measure $($i.e., any set $S$ for which the boundary of $S$ has Lebesgue measure 
equal to zero$)$.
\end{corollary}
\begin{proof}
Noting that $\eta_{\xi_{\pm}}$ are absolutely continuous with respect to Lebesgue measure, the statements follow directly from Theorem \ref{t3.13}.
\end{proof}

\begin{corollary}\lb{c3.15}
Assume Hypothesis \ref{h3.9}. If $g$ is a bounded Borel measurable function that is compactly supported and Lebesgue almost everywhere continuous on $\R$, then
\begin{equation}
\lim_{j\rightarrow \infty}\int_{\R} \xi(\lambda; H_j,H_{0,j}) d\lambda \, g(\lambda) 
= \int_{\R} \xi(\lambda;H,H_0) d\lambda \, g(\lambda).
\end{equation}
\end{corollary}
\begin{proof}
If $g$ satisfies the hypotheses of Corollary \ref{c3.15}, then choosing 
$f(\lambda):=(\lambda^2 + 1)g(\lambda)$ in \eqref{3.84} yields the result, noting that 
$f$ is a bounded ($g$ has compact support) Borel measurable function and is continuous Lebesgue-almost everywhere (and thus $\eta_{\xi_{\pm}}$-almost everywhere) on $\R$.
\end{proof}

\begin{remark}
We briefly summarize the instrumental role of convergence of the determinants, 
\eqref{3.27}, in this work.  The proof of \eqref{3.25} in Lemma \ref{l3.3} goes by simple induction on $n\in\bbN$ in \eqref{3.44}, and it is precisely \eqref{3.27} that yields the first induction step, $n=1$, in \eqref{3.44}.  Convergence of the determinants, 
\eqref{3.27}, also proves indispensable in the proof of {\it weak} convergence of the spectral shift functions, that is, \eqref{3.84}. To go from {\it vague} (cf.\ \eqref{3.74}) to \textit{weak} (cf.\ \eqref{3.84}) convergence, we simply apply the abstract Theorem 
\ref{t3.7} together with convergence of the total masses, \eqref{3.86a}.  However, it is \eqref{3.77}, and therefore \eqref{3.27}, that guarantees the requisite convergence of total masses, \eqref{3.86a}.
\end{remark}

For applications to multi-dimensional Schr\"odinger operators we need to extend the assumptions in Hypothesis \ref{h3.9} a bit:

\begin{hypothesis} \lb{h3.16}
Suppose the assumptions made in Hypothesis \ref{h3.9} with the exception of the trace class assumptions \eqref{3.4}, \eqref{3.6a}, and \eqref{3.9}. \\ 
$(ix)$ Assume that for some $($and hence for all\,$)$ $z \in \rho(H_0)$,
\begin{equation}
\ol{V_2 (H_0 - z I_{\cH})^{-1} V_1^*}, \, \ol{V_{2,j} (H_{0,j} - z I_{\cH})^{-1} V_{j,1}^*}\oplus 0, 
 \in \cB_2(\cH),   \quad j\in\bbN,   \lb{3.89}
\end{equation}
and that 
\begin{align}
\begin{split} 
& \lim_{z \downarrow -\infty} \big\|\ol{V_2 (H_0 - z I_{\cH})^{-1} V_1^*}\big\|_{\cB_2(\cH)} 
= 0.     \lb{3.90} \\
& \lim_{z \downarrow -\infty} \big\|\ol{V_{2,j} (H_{0,j} - z I_{\cH})^{-1} V_{1,j}^*}\oplus 0
\big\|_{\cB_2(\cH)} 
= 0, \quad j\in\bbN.     
\end{split} 
\end{align}
$(x)$ Suppose that for some $($and hence for all\,$)$ $z\in \C\backslash \R$, 
\begin{equation}
\lim_{j \to \infty}\big\|\big[\ol{V_{2,j}(H_{0,j} - z I_{\cH_j})^{-1}V_{1,j}^*}\oplus 0\big] 
- \ol{V_2 (H_0 - z I_{\cH})^{-1}V_1^*}\big\|_{\cB_2(\cH)} =0.     \lb{3.91} 
\end{equation} 
\end{hypothesis}

\begin{theorem} \lb{t3.17}
Assume Hypothesis \ref{h3.16}. Then the assertions of Theorem \ref{t3.10}, 
Corollary \ref{c3.11}, Theorem \ref{t3.13}, and Corollaries \ref{c3.14} and 
\ref{c3.15} hold. 
\end{theorem}
\begin{proof}
It suffices to delineate the necessary changes in the proofs due to the 
Hilbert--Schmidt hypotheses \eqref{3.89}, \eqref{3.90} as opposed to the trace class 
assumptions \eqref{3.4}, \eqref{3.6a}, and \eqref{3.9}. 

With \eqref{3.20} still valid, we abbreviate the modified perturbation determinants by 
\begin{align}
D_2(z)& = {\det}_{2,\cH}\big(I_{\cH}+\ol{V_2 (H_{0} - z I_{\cH})^{-1}V_1^*}\big), 
\quad z\in \rho(H) \cap \rho(H_0),     \no \\
D_{2,j}(z)& = {\det}_{2,\cH_j}\big(I_{\cH_j} + \ol{V_{2,j}(H_{0,j} - z I_{\cH_j})^{-1}V_{1,j}^*}\big), \quad z\in \rho(H_j) \cap \rho(H_{0,j}), \; j\in\bbN.   \lb{3.93}  
\end{align}

Consequently, by \eqref{2.28}, for $z\in\bbC\backslash\bbR$, 
\begin{align}
\frac{d}{dz}\ln D_2(z) & =\int_{\bbR}\frac{\xi(\lambda) d\lambda}{(\lambda-z)^2}  
- {\tr}_{\cH} \big((H_0 - z I_{\cH})^{-1}V(H_0 - z I_{\cH})^{-1}\big),    \no \\ 
\frac{d}{dz}\ln D_{2,j}(z) & =\int_{\bbR}\frac{\xi_j(\lambda) d\lambda}{(\lambda-z)^2}  
- {\tr}_{\cH_j} \big((H_0 - z I_{\cH_j})^{-1}V_j ({H_0} - z I_{\cH_j})^{-1}\big), \quad j\in\bbN.  
\lb{3.94}
\end{align}
In analogy to \eqref{3.28}, \eqref{3.29}, one then obtains ($a,z\in\bbC\backslash\bbR$)
\begin{align}
(z-a) \int_{\bbR}\frac{\xi(\lambda) d\lambda}{(\lambda-a)(\lambda-z)} 
& = (z-a) {\tr}_{\cH}\big(V_2 (H_0 - z I_{\cH})^{-1}\ol{(H_0 - a I_{\cH})^{-1}V_1^*} \big)\no\\
&\quad + \ln \bigg(\frac{D_2(z)}{D_2(a)} \bigg),     \lb{3.96} \\
(z-a) \int_{\bbR}\frac{\xi_j(\lambda) d\lambda}{(\lambda-a)(\lambda-z)} 
& = (z-a) \tr_{\cH_j}\big(V_{2,j} (H_{0,j} - z I_{\cH_j})^{-1} 
\ol{(H_{0,j} - a I_{\cH_j})^{-1}V_{1,j}^*} \big)\no\\
&\quad + \ln \bigg(\frac{D_{2,j}(z)}{D_{2,j}(a)}\bigg), \quad j\in\bbN.    \lb{3.97}
\end{align}
By  \eqref{3.91}, 
\begin{equation} 
\lim_{j\to\infty} D_{2,j}(z) = D_2(z),    \lb{3.98}
\end{equation}
and by \eqref{3.10}, \eqref{3.11}, 
\begin{align} 
\begin{split} 
& \lim_{j\to\infty} {\tr}_{\cH_j}\big(V_{2,j} (H_{0,j} - z)^{-1}\ol{({H_{0,j}}-a)^{-1}V_{1,j}^*} 
\big)    \lb{3.99} \\
& \quad = {\tr}_{\cH}\big(V_2 (H_0 - z I_{\cH})^{-1} \ol{(H_0 - a I_{\cH})^{-1} V_1^*} \big).
\end{split}
\end{align} 
At this point one can follow the proof of Lemma \ref{l3.3} step by step, implying the 
validity of \eqref{3.43}. In addition, \eqref{3.94}--\eqref{3.99} yield 
\begin{equation}
\lim_{j\to\infty} \int_{\bbR}\frac{\xi_j(\lambda) d\lambda}{(\lambda-a)(\lambda-z)}  
= \int_{\bbR}\frac{\xi(\lambda) d\lambda}{(\lambda-a)(\lambda-z)}, \quad 
a, z \in \bbC\backslash \bbR,    \lb{3.100} 
\end{equation}
and hence the first induction step \eqref{3.45} also holds under Hypothesis \ref{h3.16}, implying the assertions in Lemma \ref{l3.3}. The latter is the crucial input for the proof 
of Theorem \ref{t3.10} and hence for Corollary \ref{c3.11}, which both extend to the 
current Hypothesis \ref{h3.16}.

Finally, the proofs of Theorem \ref{t3.13} and Corollaries \ref{c3.14} and \ref{c3.15} 
extend without change under Hypothesis \ref{h3.16}. 
\end{proof}

\section{Applications to Schr\"odinger operators}  \lb{s4}

In our final section we briefly illustrate the applicability of Theorems \ref{t3.13} 
and \ref{t3.17} to (multi-dimensional) Schr\"odinger operators. 

\medskip 

\noindent 
$\mathbf{(I)}$ {\bf The one-dimensional case.} 

Assuming 
\begin{equation}
V \in L^1(\bbR; dx) \, \text{ real-valued},    \lb{4.1}
\end{equation}
and introducing the differential expression $\tau$ by
\begin{equation}
\tau = - \f{d^2}{dx^2} + V(x), \quad x \in J,    \lb{4.2}
\end{equation}
with $J \subseteq \bbR$ an appropriate open interval, we introduce the self-adjoint  
Schr\"odinger operator $H_{(a,b),\alpha,\beta}$ in $L^2((a,b); dx)$,  with Dirichlet boundary conditions at $x=a,b$ 
\begin{align}
& (H_{(a,b),D} f)(x) = (\tau f)(x), \quad x \in (a,b),   \no \\
& \, f \in \dom(H_{(a,b),D}) =\big\{g \in L^2((a,b); dx) \, \big|\, 
g, g' \in AC([a,b]);   \lb{4.3} \\ 
& \hspace*{3.55cm} g(a)=g(b)=0; \, \tau f \in L^2((a,b); dx)\big\},    \no 
\end{align}
and the self-adjoint Schr\"odinger operator $H$ in $L^2(\bbR; dx)$, 
\begin{align}
\begin{split} 
& (H f)(x) = (\tau f)(x), \quad x \in \bbR,    \\
& \, f \in \dom(H) =\big\{g \in L^2(\bbR; dx) \, \big| \, 
g, g' \in AC_{\loc}(\bbR);  \, \tau f \in L^2(\bbR; dx)\big\}.     \lb{4.4}
\end{split}
\end{align}
Here $AC([a,b])$ (resp., $AC_{\loc}(\bbR)$) abbreviates the set of absolutely 
continuous functions on $[a,b]$ (resp., the set of locally absolutely continuous 
functions on $\bbR$).  

In the special case where $V=0$ a.e.\ on $\bbR$, the operators in \eqref{4.3} and 
\eqref{4.4} are denoted by $H_{0,(a,b),D}$ and $H_0$, respectively. 

The Dirichlet Green's functions (i.e., the integral kernels of the resolvents) associated 
with  $H_{0,(a,b),D}$ and $H_0$ are then given by 
\begin{align}
& G_{0,(a,b),D}(z,x,x') = (H_{0,(a,b),D} - z I_{(a,b)})^{-1} (x,x')  \no \\
& \quad = \f{1}{z^{1/2}\sin(z^{1/2}(b-a))} \begin{cases} 
\sin(z^{1/2}(x-a)) \sin(z^{1/2}(b-x')), & a \leq x \leq x' \leq b, \\
\sin(z^{1/2}(x'-a)) \sin(z^{1/2}(b-x)), & a \leq x' \leq x \leq b,
\end{cases}    \no \\
& \hspace*{7cm} z \in\bbC \backslash\big\{n^2 \pi^2 (b-a)^{-2}\big\}_{n\in\bbN},   \lb{4.5} 
\end{align}
and 
\begin{align} 
\begin{split} 
G_{0}(z,x,x') = (H_0 - z I_{\bbR})^{-1} (x,x') 
= \f{i}{2 z^{1/2}} e^{i z^{1/2} |x-x'|}, \quad  x, x' \in \bbR,&   \lb{4.6} \\
z \in \bbC\backslash [0,\infty), \; \Im(z^{1/2}) \geq 0,&   
\end{split} 
\end{align}
respectively, where $I_J$ denotes the identity operator in $L^2(J; dx)$ for 
$J\subseteq \bbR$ an interval. 

Moreover, we also recall the integral kernels for the square root of resolvents 
(cf.\ \cite[p.\ 325]{GH03} and \cite{GN11}),  
\begin{align} \lb{4.7}
& R_{0,(a,b),D}^{1/2}(z,x,x') = (H_{0,(a,b),D} - z I_{(a,b)})^{-1/2} (x,x')     \no \\
& \quad = \f{1}{\pi} \int_0^\infty dt \, t^{-1/2} G_{0,(a,b),D}(z-t,x,x'),     \\ 
& \hspace*{.1cm} z \in\bbC \backslash\big\{n^2 \pi^2 (b-a)^{-2}\big\}_{n\in\bbN}, 
\; x, x' \in [a,b],   \no 
\end{align}
and  
\begin{align} \lb{4.8}
\begin{split} 
R_0^{1/2}(z,x,x') = (H_0 - z I_{\bbR})^{-1/2} (x,x') 
= \pi^{-1} H_0^{(1)} (z^{1/2}|x-x'|),&    \\
z \in \bbC\backslash [0,\infty), \; \Im(z^{1/2}) \geq 0, \; x, x' \in \bbR,& 
\end{split} 
\end{align}
where $H_0^{(1)}(\cdot)$ denotes the Hankel function of the first kind and order zero 
(cf.\ \cite[Sect.\ 9.1]{AS72}). Moreover, employing domain monotonicity for Dirichlet Green's functions (see, e.g., \cite[Sect.\ 1.VII.6]{Do84}), that is, 
\begin{align}
\begin{split}
0 \leq G_{0,(a,b),D}(-E,x,x') \leq G_{0,(a',b'),D}(-E,x,x') \leq
G_0 (-E,x,x'),&   \\
x, x' \in [a,b] \subseteq [a',b'], \; E>0,&     \lb{4.10}
\end{split} 
\end{align}
and inserting \eqref{4.10} into \eqref{4.7} yields
\begin{equation}\lb{4.11}
0 \leq R_{0,(a,b),D}^{1/2} (-E,x,x') \leq 
R_0^{1/2} (-E,x,x'), \quad x,x'\in [a,b], \; E>0.
\end{equation}
In fact, the upper bound in \eqref{4.10} can be made more explicit by noting
\begin{align}
0 & \leq G_{0,(a,b),D}(-E,x,x')   \no \\
& = \f{1}{2 E^{1/2}} 
\f{\big[e^{E^{1/2}(x-a)} - e^{-E^{1/2}(x-a)}\big]\big[e^{E^{1/2}(b-x')} 
- e^{-E^{1/2}(b-x')}\big]}{e^{E^{1/2}(b-a)} - e^{-E^{1/2}(b-a)}}   \no \\
& = G_0 (-E,x,x') - \f{1}{2 E^{1/2}}  
\f{\big[e^{E^{1/2}(x+x'-b-a)} - e^{E^{1/2}(a+x'-b-x)}\big]}
{e^{E^{1/2}(b-a)} - e^{-E^{1/2}(b-a)}}     \no \\
& \quad - \f{1}{2 E^{1/2}} 
\f{\big[e^{E^{1/2}(b+a-x-x')} - e^{E^{1/2}(x+a-b-x')}\big]}
{e^{E^{1/2}(b-a)} - e^{-E^{1/2}(b-a)}},    \no \\ 
& \leq G_0 (-E,x,x'),   \quad a \leq x \leq x' \leq b,\lb{4.12}
\end{align}
and analogously for $a \leq x'\leq x \leq b$. (One verifies that all square brackets in 
\eqref{4.11} are nonnegative.) The relations in \eqref{4.12} also show that 
\begin{equation}
\lim_{a \downarrow -\infty, \, b \uparrow \infty} G_{0,(a,b),D}(-E,x,x') = G_0 (-E,x,x') 
\, \text{ pointwise,}    \lb{4.13}
\end{equation}
that is, for fixed $E<0$ and $x, x' \in\bbR$. (The latter is easily seen to extend to all 
fixed $z\in\bbC\backslash\bbR$.)

Next, one factors $V$ as
\begin{equation}
V(x)=u(x)v(x),\quad v(x)=|V(x)|^{1/2}, \quad u(x)=v(x) \, \sgn(V(x)), \quad 
x \in \bbR,    \lb{4.14}
\end{equation}
and introduces  
\begin{equation}
V_{(a,b)}(x) = V(x)|_{(a,b)}, \quad v_{(a,b)}(x) = v(x)|_{(a,b)}, \quad 
u_{(a,b)}(x) = u(x)|_{(a,b)}, \quad x \in (a,b).    \lb{4.15}
\end{equation}
Then in the notation employed in Section \ref{s3}, 
\begin{align}
& \text{$u$ corresponds to $V_2$, \;\; $v$ corresponds to $V_1^*$} \\
& \text{$u_{(a,b)}$ corresponds to $V_{2,j}$, \;\; 
$v_{(a,b)}$ corresponds to $V_{1,j}^*$,} \\
& \text{$a\downarrow -\infty$, $b\uparrow \infty$ corresponds to $j \to \infty$, etc.} 
\end{align}
Moreover, the estimate 
\begin{equation}
0 \leq \big|H_0^{(1)}(x)\big| \leq C \, \ln\bigg(\frac{ex}{1+x}\bigg) 
\frac{e^{-x}}{2\pi x^{1/2}+1},  \quad  x>0,    \lb{4.17a}
\end{equation}
for a suitable constant $C>0$ (cf.\ \cite[Sect.\ 9.6]{AS72} for the proper asymptotic 
relations as $x\downarrow 0$ and $x \to\infty$, implying \eqref{4.17a}) then readily 
proves that
\begin{align} 
& u (H_0 + I_{\bbR})^{-1/2}, \, u (H_0 + I_{\bbR})^{-1} \in \cB_2\big(L^2(\bbR; dx)\big),  
\lb{4.17b} \\
& \ol{u (H_0 + I_{\bbR})^{-1}v} = \big[u (H_0 + I_{\bbR})^{-1/2}\big] 
\big[v (H_0 + I_{\bbR})^{-1/2}\big]^* \in \cB_1\big(L^2(\bbR; dx)\big),   \lb{4.17c} \\  
& \big[(H + I_{\bbR})^{-1}-(H_0 + I_{\bbR})^{-1}\big] \in \cB_1(L^2(\bbR;dx)).  \lb{4.17d}
\end{align} 

At this point it is possible to verify that each item in Hypothesis \ref{h3.9} applies 
and hence that Theorem \ref{t3.10}, Corollary \ref{c3.11}, Theorem \ref{t3.13}, and 
Corollaries \ref{c3.14} and \ref{c3.15} all hold in the context of Dirichlet boundary 
conditions at $x = a, b$. 

The case of a half-line with $\bbR$ replaced by $[0,\infty)$ and 
$(a,b)$ by $(0,R)$, $R>0$, has been dealt with in great detail in \cite{GN11} 
(in part, using techniques developed in \cite{GM03}). In 
\cite{GN11} also the case of all separated self-adjoint boundary conditions was discussed 
in depth, by invoking Krein-type resolvent equations that reduce general separated boundary conditions to the case of Dirichlet boundary conditions. This strategy 
also applies to all separated self-adjoint boundary conditions in the current case of 
$\bbR$ and $(a,b)$; we omit further details at this point. (See also 
\cite[Chs.\ 2,3,5,9]{Sc81} for a detailed treatment of the case $n=1$.) 

\medskip

\noindent 
$\mathbf{(II)}$ {\bf The two- and three-dimensional case (with Dirichlet boundary conditions).} 

Assuming 
\begin{align}
& V\in \cR_{2,\delta} \, \text{ for some $\delta>0$, real-valued, if $n=2$,}   \lb{4.18} \\
& V\in \cR_3 \cap L^1(\bbR^3;d^3x), \, \text{ real-valued, if $n=3$,}    \lb{4.19}
\end{align}
where 
\begin{align}
\cR_{2,\delta}&=\big\{V:\bbR^2\rightarrow \bbC, \, \text{measurable} 
\,\big|\, V^{1+\delta},(1+|\cdot|^{\delta})V\in L^1(\bbR^2;d^2 x)\big\},    \lb{4.20} \\
\cR_3&=\bigg\{V:\bbR^3\rightarrow \bbC, \, \text{measurable} \,\bigg|\,  
\int_{\bbR^6}d^3x \, d^3x' \, \f{|V(x)||V(x')|}{|x-x'|^2} <\infty\bigg\},  
\lb{4.21}
\end{align}
(with $\cR_3$ the set of {\it Rollnik} potentials), we introduce the differential expression
\begin{equation}
\tau = - \Delta + V(x), \quad x \in \Omega,    \lb{4.22}
\end{equation}
with $\Omega \subseteq \bbR^n$, $n=2,3$, an appropriate open set. 

Denoting by $B_R\subseteq \bbR^n$, $n=2,3$, the open ball of radius $R>0$, centered 
at the origin, we now introduce the self-adjoint Schr\"{o}dinger operator $H_{B_R,D}$ in 
$L^2(B_R; d^n x)$, $n=2,3$, with {\it Dirichlet} boundary conditions at $\partial B_R$, by
\begin{align}
\begin{split}
& (H_{B_R,D} f)(x) = (\tau f)(x), \quad x \in B_R,    \\
& \, f \in \dom(H_{B_R,D}) = \big\{g \in L^2(B_R; d^n x) \,\big| \, g \in 
H^1_0 (B_R); \, \tau g \in L^2(B_R; d^n x)\big\},  \lb{4.23} 
\end{split}
\end{align}
and  the self-adjoint Schr\"{o}dinger operator $H$ in $L^2(\bbR^n; d^n x)$, $n=2,3$, by
\begin{align}
\begin{split}
& (H f)(x) = (\tau f)(x), \quad x \in \bbR^n,    \\
& \, f \in \dom(H) = \big\{g \in L^2(\bbR^n; d^n x) \,\big| \, g \in 
H^1 (\bbR^n); \, \tau g \in L^2(\bbR^n; d^n x)\big\}.   \lb{4.24} 
\end{split}
\end{align} 
Here $\tau f = - \Delta f + V f$ is interpreted in the sense of distributions (i.e., in 
$\cD'(B_R)$, resp., $\cD'(\bbR^n)$, $n=2,3$).

In the special case where $V=0$ a.e.\ on $\bbR^n$, the operators in \eqref{4.23} and 
\eqref{4.24} are denoted by $H_{0,B_R,D}$ and $H_0$, respectively, and given by
\begin{align}
& H_{0,B_R,D} = - \Delta,  \quad 
\dom(H_{0,B_R,D}) = \big\{g \in L^2(B_R; d^n x) \,\big| \, g \in 
H^{1}_0 (B_R) \cap H^{2} (B_R);    \no \\
& \hspace*{7.4cm} - \Delta g \in L^2(B_R; d^n x)\big\},    \lb{4.25} \\
& H_0 = - \Delta, \quad  \dom(H_0) = H^2 (\bbR^n).       \lb{4.26} 
\end{align} 

The method of images (cf.\ \cite[p.\ 264]{CH89}) then permits one to explicitly compute the Dirichlet Green's function for the ball $B_R$ as 
\begin{align}
& G_{0,B_R,D} (z,x,x') = (H_{0,B_R,D} - z I_{B_R})^{-1}(x,x')  \no \\ 
& \quad = \psi_n(z,|x-x'|) - \psi_n\bigg(z,\frac{|x'|}{R}\bigg|x-\frac{R^2}{|x'|^2}x' 
\bigg|\bigg),  \lb{4.27} \\
& \hspace*{2.15cm}  z \in \bbC \backslash  [0,\infty), \; x \neq x', \; x, x' \in B_R,   \no 
\end{align}
where
\begin{equation}
\psi_n(z,r) = \begin{cases}
(i/4) H^{(1)}_0(z^{1/2} r), & n=2, \\
e^{i z^{1/2} r}/[4\pi r], & n=3, 
\end{cases}  \quad z \in \bbC \backslash  [0,\infty), \; \Im(z^{1/2}) \geq 0, \;  r > 0,   \no 
\end{equation}
with $H^{(1)}_0(\cdot)$ again the Hankel function of the 1st kind and order zero.  
Similarly, 
\begin{align}
G_0(z,x,x') &= (H_0 - z I_{\bbR^n})^{-1}(x,x') = \psi_n(z,|x-x'|)   \no \\ 
& = \begin{cases}
(i/4) H^{(1)}_0(z^{1/2} |x-x'|), & n=2, \\
e^{i z^{1/2} |x-x'|}/[4\pi |x-x'|], & n=3, 
\end{cases}   \lb{4.28} \\
& \hspace*{-1.6cm} z \in \bbC \backslash  [0,\infty), \; \Im(z^{1/2}) \geq 0, \; 
x \neq x', \; x, x' \in \bbR^n, \; n=2,3.  \no 
\end{align}
Here $I_{\Omega}$ denotes the identity operator in $L^2(\Omega; d^n x)$ for 
$\Omega \subseteq \bbR^n$ an open set. Consequently, 
\begin{equation}
\big|{ G_{0,B_R,D} (z,x,x')}\big| \leq C \psi_n (z,|x - x'|), \quad  
x \neq x', \; x, x' \in \bbR^n, \; n=2,3,    \lb{4.29}
\end{equation}
for some constant $C>0$. Moreover, one gets as a pointwise limit
\begin{align}
\lim_{R\to\infty} G_{0,B_R,D} (z,x,x') = G_0(z,x,x'), \quad 
z \in \bbC \backslash  [0,\infty), \; x \neq x', \; x, x' \in \bbR^n. 
\end{align} 

Next, one again factors $V$ as
\begin{equation}
V(x)=u(x)v(x),\quad v(x)=|V(x)|^{1/2}, \quad u(x)=v(x) \, \sgn(V(x)), \quad 
x \in \bbR^n,    \lb{4.35}
\end{equation}
and introduces  
\begin{equation}
V_{B_R}(x) = V(x)|_{B_R}, \quad v_{B_R}(x) = v(x)|_{B_R}, \quad 
u_{B_R}(x) = u(x)|_{B_R}, \quad x \in B_R.    \lb{4.36}
\end{equation}
Then in the notation employed in Section \ref{s3}, 
\begin{align}
& \text{$u$ corresponds to $V_2$, \;\; $v$ corresponds to $V_1^*$} \\
& \text{$u_{B_R}$ corresponds to $V_{2,j}$, \;\; 
$v_{B_R}$ corresponds to $V_{1,j}^*$,} \\
& \text{$B_R \to \bbR^n$ as $R\uparrow \infty$ corresponds to $j \to \infty$, etc.,} 
\end{align}
and using again the estimate \eqref{4.17a} in the case $n=2$, one concludes that 
for $n=2,3$, 
\begin{align}
& u (H_0 + I_{\bbR^n})^{-1/2} \in \cB_4\big(L^2(\bbR^n; d^n x)\big), \quad 
u (H_0 + I_{\bbR^n})^{-1} \in \cB_2\big(L^2(\bbR^n; d^n x)\big),   \lb{4.40} \\
& \ol{u (H_0 + I_{\bbR^n})^{-1}v} = \big[u (H_0 + I_{\bbR^n})^{-1/2}\big] 
\big[v (H_0 + I_{\bbR^n})^{-1/2}\big]^* \in \cB_2\big(L^2(\bbR^n; d^n x)\big),   \lb{4.41}  \\
& \big[(H + I_{\bbR^n})^{-1}-(H_0 + I_{\bbR^n})^{-1}\big] \in \cB_1(L^2(\bbR^n;d^n x)),  
\lb{4.42} \\
& \ol{(H_0 + I_{\bbR^n})^{-1}V(H_0 + I_{\bbR^n})^{-1}}    \no \\
& \quad = \big[u (H_0 + I_{\bbR^n})^{-1}\big]^* \big[v (H_0 + I_{\bbR^n})^{-1}\big] 
\in \cB_1(L^2(\bbR^n;d^n x)).    \lb{4.43}
\end{align} 
More precisely, the fact that 
$u (H_0 + I_{\bbR^n})^{-1} \in \cB_2\big(L^2(\bbR^n; d^n x)\big)$ follows from an 
application of \cite[Theorem\ 4.1]{Si05}, and thus  
$u (H_0 + I_{\bbR^n})^{-1/2} \in \cB_4\big(L^2(\bbR^n; d^n x)\big)$ follows from the 
fact that $T \in \cB_4(\cH)$ if and only if $T^*T \in \cB_2(\cH)$. Indeed, one chooses for 
$T$ the operator $u(H_0 + I_{\bbR^n})^{-1/2}$ and defines the self-adjoint and unitary 
operator $S$ of multiplication by $\sgn(V(x))$ with $\sgn(V(x)) =1$ for $V(x) \geq 0$ a.e. 
and $\sgn(V(x)) = -1$ for $V(x) < 0$ a.e. 

Moreover, an application of 
\cite[Proposition\ 4.4]{Si05} then also proves that 
\begin{align}
\begin{split} 
 & u (H_0 + I_{\bbR^n})^{-1/2} \notin \cB_2\big(L^2(\bbR^n; d^n x)\big), \\
& \ol{u (H_0 + I_{\bbR^n})^{-1}v} \notin \cB_1\big(L^2(\bbR^n; d^n x)\big), \quad 
 n=2,3.  
 \end{split}
\end{align} 
Indeed, since $T^* T \in \cB_1(\cH)$ if and only if $T \in \cB_2(\cH)$, again choosing 
for $T$ the operator $u(H_0 + I_{\bbR^n})^{-1/2}$ proves that 
$u(H_0 + I_{\bbR^n})^{-1/2} \notin \cB_2\big(L^2(\bbR; d^n x)\big)$ since 
$(|p|^2 + 1)^{-1} \notin L^2(\bbR; d^n p)$ for $n=2,3$. 

This illustrates that even though \eqref{4.43} holds in dimensions $n = 2,3$ (just like 
it holds for $n = 1$, cf.\ \eqref{4.17d}), the use of the $2$-modified Fredholm determinant 
${\det}_{2, L^2(\bbR^n; d^n x)}(\cdot)$ in connection with formulas of the type 
\eqref{3.94} is inevitable in dimensions $n = 2, 3$ (as opposed to the case $n=1$). Thus, 
Hypothesis \ref{h3.9} needed to be extended to Hypothesis \ref{h3.16} in order to be 
able to handle the multi-dimensional cases $n = 2,3$. 

We also note in connection with \eqref{2.21} that 
\begin{align}
\eta' (z) &= {\tr}_{L^2(\bbR^n; d^n x)} 
\big(\ol{(H_0 - z I_{\bbR^n})^{-1}V(H_0 - z I_{\bbR^n})^{-1}}\big)     \no \\  
&= \begin{cases}
- \f{1}{4 \pi z} \int_{\bbR^2} d^2 x \, V(x), & n=2, \\
\f{i}{8 \pi z^{1/2}} \int_{\bbR^3} d^3 x V(x), & n=3,  
\end{cases}   
\end{align}
and hence (cf.\ also \cite{GLMZ05}),
\begin{align}
\xi(\lambda; H, H_0) &= \f{1}{\pi} \Im\big(\ln\big({\det}_{2, L^2(\bbR^n; d^n x)} 
\big(I_{\bbR^n} + \overline{u (H_0 - (\lambda + i0) I_{\bbR^n})^{-1} v}\big)\big)\big)    \no \\
& \quad + \begin{cases}
\f{1}{4 \pi} \int_{\bbR^2} d^2 x \, V(x), & \lambda > 0, \; n=2, \\[1mm]
\f{\lambda^{1/2}}{4 \pi^2} \int_{\bbR^3} d^3 x \, V(x), & \lambda > 0, \; n=3, \\[1mm]
0, & \lambda < 0, \; n=2,3, 
\end{cases}
\end{align}
in accordance with the normalization \eqref{2.9}. 

At this point it is possible to verify that each item in Hypothesis \ref{h3.16} applies 
and hence that Theorem \ref{t3.10}, Corollary \ref{c3.11}, Theorem \ref{t3.13}, and 
Corollaries \ref{c3.14} and \ref{c3.15} all hold in the context of Dirichlet boundary 
conditions at $\partial B_R$ for $n=2,3$. (We also refer to \cite{BGD88}, \cite{Ch84}, 
and the references cited therein, in the case $n = 2$, and to the detailed treatment of 
the case $n = 3$ in \cite[Chs.\ I--IV]{Si71}.) 

\medskip

\noindent 
$\mathbf{(III)}$ {\bf Possible generalizations.} 

Without providing full details, we will hint at various possible extensions including more general regions and other boundary conditions. 

In this context we find it convenient to recall some facts on positivity preserving (resp., 
improving) operators. Suppose $(X, \cA, \mu)$ is a $\sigma$-finite measure space and 
$\cK = L^2(X; d\mu)$ a complex, separable Hilbert space. Then $A \in\cB(\cK)$ is called 
{\it positivity preserving} (resp., {\it positivity improving}) if
\begin{equation}
0 \neq f \in \cK, \, f \geq 0 \, \text{$\mu$-a.e.} \text{ implies } \, A f \geq 0 \, 
\text{ (resp.\ $Af > 0$) $\mu$-a.e.}
\end{equation}
(We refer, e.g., to \cite{BKR80}, \cite{Da73}, \cite[Ch.\ 7]{Da80}, \cite{Fa72}, 
\cite[Sect.\ 8]{Fa75}, \cite{FS75}, 
\cite{Go77}, \cite{HSU77}, \cite{KR80}, \cite{Kr64}, \cite{KR50}, \cite{MO86}, \cite{Pe81}, 
\cite[Sct.\ XIII.12]{RS78}, \cite{Si73}, \cite{Si77a}, \cite{Si79}, and the references cited 
therein for the basics of this subject.) Positivity preserving (resp., improving) of $A$ will be denoted by 
\begin{equation}
A \succcurlyeq 0 \, \text{ (resp., $A \succ 0$).}
\end{equation}
(or by $0 \preccurlyeq A$ (resp., $0 \prec A$)). Similarly, if $A, B \in \cB(\cK)$, then  
\begin{equation}
A \succcurlyeq B \succcurlyeq 0 \, \text{ (resp., $A \succ B \succ 0$)}
\end{equation}
(or $0 \preccurlyeq B \preccurlyeq A$ (resp., $0 \prec B \prec A$)) imply that 
$A$, $B$, and $A - B$ are positivity preserving (resp., positivity improving). 

Considering the contraction semigroup $T(t) = e^{-t H}$, $ t\geq 0$, with $H\geq 0$ 
self-adjoint in $\cK$, one uses well-known relations
\begin{align}
& (H + \lambda I_{\cK})^{-1} = \int_0^{\infty} dt \, e^{- t \lambda} e ^{- t H}, 
\quad \lambda > 0,   \\
& \, e^{-t H} = \slim_{n \to \infty} \big[(t/n) H + I_{\cK}\big]^{-n}, \quad t \geq 0,
\end{align} 
to prove that 
\begin{align}
\begin{split}
& e^{-tH} \, \text{ is positivity preserving for all $t \geq 0$ if and only if} \\
& \quad (H + \lambda I_{\cK})^{-1} \, \text{ is positivity preserving for all $\lambda > 0$.}
\end{split} 
\end{align}
Analogous statements hold if $H\geq 0$ is replaced by $H \geq c I_{\cK}$ for some 
$c \in \bbR$.

Moreover, we note that if $A$ is an integral operator with integral kernel $A(x,x')$ 
for $d\mu \otimes d\mu$-a.e.\ $x, x' \in X$, and assuming $A(\cdot,\cdot) \in L^1(X\times X; d\mu \otimes d\mu)$, then  
\begin{equation}
A \succcurlyeq 0  \, \text{ if and only if } \, A(x,x') \geq 0 \, 
\text{ for $d\mu \otimes d\mu$-a.e.\ $x, x' \in X$,}
\end{equation}
with $d\mu \otimes d\mu$ denoting the product measure on $X \times X$. Clearly, 
$A \succ 0$ if $A(\cdot,\cdot) > 0$ $d\mu \otimes d\mu$-a.e. 

We also remark that these notions of positivity preserving (resp., improving) naturally extend 
to a two-Hilbert space setting in which one deals with a second Hilbert space $L^2(Y; d\nu)$ with $Y \subset X$ and $\wti \mu = \mu|_{Y}$, see, for instance, \cite{BKR80}, \cite{KR80}. This is also frequently done in connection with (nondensely defined) quadratic forms (cf., e.g., \cite[p.\ 61--62]{Da89}), and we will employ this notation in the following without further comment.  

First, we turn to the case of Dirichlet boundary conditions for general nonempty, open,  bounded sets $\Omega\subset \bbR^n$, $n \in \bbN$, rather than just balls 
$B_R \subset \bbR^n$, $R > 0$. In this context one recalls that the Dirichlet Laplacian 
$- \Delta_{\Omega,D}$ in $L^2(\Omega; d^n x)$, by definition, is the uniquely associated self-adjoint and strictly positive operator with the closure of the sesquilinear form 
\begin{equation} 
\int_{\Omega} d^n x \, \ol{(\nabla u)(x)} \cdot (\nabla v)(x),  
\quad f, g \in C_0^\infty(\Omega), 
\end{equation} 
with domain the Sobolev space $H^1_0(\Omega)$ (see, e.g., \cite[Sect.\ 1.8]{Da89}, 
\cite[Ch.\ VII]{EE89}, \cite[Sect.\ XIII.15]{RS78}). Domain monotonicity for the Dirichlet Laplacian then takes on the form (cf., \cite{BAH01}, \cite[Sect.\ 2.1]{Da89}, 
\cite[App.\ 1 to Sect.\ XIII.12]{RS78}) 
\begin{equation}
0 \preccurlyeq e^{- t (- \Delta_{\Omega_1,D})} 
\preccurlyeq e^{- t (- \Delta_{\Omega_2,D})}, \quad t \geq 0,  \lb{4.55}
\end{equation} 
or equivalently,
\begin{equation}
0 \preccurlyeq (- \Delta_{\Omega_1,D} + \lambda I_{\Omega_1})^{-1} 
\preccurlyeq (- \Delta_{\Omega_2,D} + \lambda I_{\Omega_2})^{-1} ,    \lb{4.56}
\quad \lambda > 0,  
\end{equation} 
assuming $\Omega_1 \subseteq \Omega_2 \subset \bbR^n$, $\Omega_j$ nonempty, 
open, and bounded, $j=1,2$. (Positivity preserving in \eqref{4.55} and \eqref{4.56} actually extends to positivity improving if $\Omega_j$ are connected, $j=1,2$, and 
$\Omega_1$ is strictly contained in $\Omega_2$ such that 
$- \Delta_{\Omega_1,D} \neq - \Delta_{\Omega_2,D}$, cf.\ \cite{KR80}).  
In particular, this yields the domain monotonicity of heat kernels,
\begin{equation}
0 \leq e^{- t (- \Delta_{\Omega_1,D})} (x,x') 
\leq e^{- t (- \Delta_{\Omega_2,D})} (x,x'), \quad t \geq 0, \; x, x' \in \Omega_1 
\subseteq \Omega_2,    \lb{4.57}
\end{equation} 
or equivalently, the domain monotonicity of Green's functions for Dirichlet Laplacians, 
\begin{equation}
0 \leq G_{0, \Omega_1,D} (- \lambda,x,x') 
\leq G_{0, \Omega_2,D} (- \lambda, x,x'), \quad \lambda > 0, \; x, x' \in \Omega_1 
\subseteq \Omega_2, \; x \neq x',    \lb{4.58} 
\end{equation} 
where $G_{0,\Omega,D}(z,x,x') = (- \Delta_{\Omega,D} - z I_{\Omega})^{-1}(x,x')$, 
$x, x' \in \Omega$, denotes the Green's function of (i.e., the integral kernel of the resolvent of) the Dirichlet Laplacian $- \Delta_{\Omega,D}$ in $L^2(\Omega; d^n x)$, $n \in \bbN$. 
We note that domain monotonicity of heat kernels as in \eqref{4.57} also follows from 
their representation in terms of Wiener measure (see, e.g., \cite{FS75}, \cite{Si78}).  

Gaussian upper and lower bounds on heat kernels have been studied very extensively in the literature (see, e.g., \cite{AtE97}, \cite{Da87}, \cite[Ch.\ 3]{Da89}, \cite{vdB90}, and the references therein). Here we just mention the rough Green's function estimate based on domain monotonicity, that is, on $\Omega \subset \bbR^n$ (see also 
\cite[Sect.\ 1.2]{Ke94}, \cite{vdB90}),
\begin{align}
& G_{0,\Omega,D}(-\lambda,x,x') \leq G_0 (-\lambda, x,x')    \no \\
& \quad = \f{1}{2 \pi} \bigg(\f{2\pi |x - x'|}{\lambda^{1/2}}\bigg)^{(2-n)/2} 
K_{(n-2)/2} (\lambda^{1/2} |x - x'|)     \no \\
& \quad \leq \begin{cases} C_{\lambda,\Omega,n} |x - x'|^{2-n}, & n \geq 3, \\
C_{\lambda,\Omega} \big|\ln\big(1 + |x-x'|^{-1}\big)\big|, & n=2,  \end{cases} 
\quad \lambda >0, \; x, x' \in \Omega, \; x \neq x',     \lb{4.59}
\end{align}
with $K_{\alpha}(\cdot)$ the modified irregular Bessel function of order $\alpha$ 
(cf.\ \cite[Sect.\ 9.6]{AS72}) and 
\begin{equation}
G_0(z,x,x') = (H_0 -z I_{\bbR^n})^{-1}(x,x'), \quad z \in \bbC\backslash [0,\infty), 
\; x, x' \in\bbR^n, \; x \neq x', \; n \in \bbN,
\end{equation}
the Green's function of the self-adjoint realization of $-\Delta$ in $L^2(\bbR^n; d^n x)$, 
\begin{equation}
H_0 = - \Delta, \quad \dom(H_0) = H^2(\bbR^n).  
\end{equation}
The estimate \eqref{4.59} ignores all effects of the boundary $\partial\Omega$ of 
$\Omega$, but it suffices for the purpose at hand (cf.\ \eqref{4.17a}, \eqref{4.28}). 

It is worth noting that these considerations extend to Schr\"odinger operators defined as 
form sums $- \Delta_{\Omega,D} +_q V$ in $L^2(\Omega; d^n x)$, assuming 
$0 \leq V \in L^1_{\loc}(\Omega; d^n x)$, $n \in \bbN$ (and also to additional, appropriately relatively form bounded, potentials), see \cite[Chs.\ 1--3]{Da89}. In addition, a 
Feynman--Kac approach to semigroups can be applied as long as $V_\pm$ belong 
to appropriate Kato classes (cf.\ \cite{AS82}, \cite{CZ95}, \cite{DvC00}). 

The case of Neumann boundary conditions, and more generally, that of Robin boundary conditions, is a bit more involved as domain monotonicity does not hold even for general 
convex domains (cf.\ \cite{BB93}, \cite{Hs94}) and a certain regularity of the boundary 
$\partial\Omega$ of $\Omega$ needs to be assumed. In addition, reflecting Brownian 
motion only works for special and sufficiently regular domains $\Omega \subset \bbR^n$. Still, one can proceed along the following lines, assuming $\Omega$ to be bounded and smooth, for simplicity. (The case of minimally smooth, that is, bounded Lipschitz domains, will be considered elsewhere \cite{GMN11}.) 

The Neumann sesquilinear form in $L^2(\Omega; d^nx)$ is given by 
\begin{equation}
\int_{\Omega} d^n x \, \ol{(\nabla u)(x)} \cdot (\nabla v)(x), \quad u,v \in H^1(\Omega),
\end{equation}
and its uniquely associated self-adjoint and nonnegative operator in 
$L^2(\Omega; d^n x)$ represents the Neumann Laplacian $- \Delta_{\Omega,N}$. The coresponding Neumann boundary condition then reads 
$(\partial u/\partial \nu)|_{\partial \Omega} = 0$, $u \in \dom(- \Delta_{\Omega,N})$, 
with $\nu$ the normal unit vector to $\partial \Omega$ and $\partial/\partial \nu$ 
denoting the normal derivative. 

Similarly, for $\theta \in C(\partial\Omega)$, the Robin sesquilinear form is of the type 
\begin{equation} 
\int_{\Omega} d^n x \,\ol{(\nabla u)(x)}\cdot(\nabla v)(x)
+ \int_{\partial\Omega} d^{n-1} \omega (\xi) \, \theta (\xi) \ol{u(\xi)} v(\xi),
\quad u, v \in H^1(\Omega),
\end{equation}
with $d^{n-1} \omega$ the surface measure on $\partial \Omega$ (cf.\ \cite{GM09} for 
more details). The uniquely associated self-adjoint operator in $L^2(\Omega; d^n x)$ then represents the Robin Laplacian $- \Delta_{\Omega,\theta}$. The corresponding Robin boundary condition is of the form 
 $(\partial u/\partial \nu)|_{\partial \Omega} + \theta \, u|_{\partial \Omega} = 0$, 
 $u \in \dom(- \Delta_{\Omega,\theta})$. 

Assuming $0 \leq \theta_1(\xi) \leq \theta_2(\xi)$, $\xi \in \partial\Omega$ and 
$\Omega_1 \subseteq \Omega_2$ one then has the positivity preserving relations 
proved in \cite{BKR80}, \cite{KR80} (see also \cite[p.\ 22]{BS05})
\begin{align}
\begin{split} 
0 & \preccurlyeq e^{- t (- \Delta_{\Omega_1,D})} 
\preccurlyeq e^{- t (- \Delta_{\Omega_2,D})}    \\ 
& \preccurlyeq e^{- t (- \Delta_{\Omega_2,\theta_2})} 
\preccurlyeq e^{- t (- \Delta_{\Omega_2,\theta_1})}
\preccurlyeq e^{- t (- \Delta_{\Omega_2,N})}, 
\quad t \geq 0,    \lb{4.64}
\end{split} 
\end{align} 
or equivalently,
\begin{align}
0 & \preccurlyeq (- \Delta_{\Omega_1,D} + \lambda I_{\Omega_1})^{-1} 
\preccurlyeq (- \Delta_{\Omega_2,D} + \lambda I_{\Omega_2})^{-1}   \no \\
& \preccurlyeq (- \Delta_{\Omega_2,\theta_2} + \lambda I_{\Omega_2})^{-1} 
\preccurlyeq (- \Delta_{\Omega_2,\theta_1} + \lambda I_{\Omega_2})^{-1}
\preccurlyeq (- \Delta_{\Omega_2,N} + \lambda I_{\Omega_2})^{-1}, 
\quad \lambda > 0.     \lb{4.65}
\end{align}  
(Again, positivity preserving in \eqref{4.64} and \eqref{4.65} actually extends to positivity improving if $\Omega_j$ are connected, $j=1,2$, $\Omega_1$ is strictly contained in 
$\Omega_2$ such that $- \Delta_{\Omega_1,D} \neq - \Delta_{\Omega_2,D}$, and 
$\theta_1 \neq \theta_2$ such that 
$- \Delta_{\Omega_2,\theta_1} \neq - \Delta_{\Omega_2,\theta_2}$, cf.\ \cite{KR80}).

Relations \eqref{4.64} and \eqref{4.65} then yield analogous pointwise bounds on heat kernels and Green's functions to those in \eqref{4.57} and \eqref{4.58} and we note 
that Gaussian upper bounds for Neumann heat kernels are 
available in the literature (see, e.g., \cite{AtE97}, \cite[Theorem\ 3.2.9]{Da89}).

We conclude by noting that in dimensions $n \geq 3$, suitably higher modified Fredholm determinants ${\det}_{p, L^2(\bbR^n; d^n x)}(\cdot)$, $p = p(n) \in\bbN$, must be applied.
This is discussed in some detail in \cite{Ya07}, \cite[Ch.\ 9]{Ya10}, and under somewhat different assumptions on $V$ (sign definiteness of $V$, but otherwise more general 
local singularities of $V$ are permitted) in \cite[Sect.\ 1.6]{Ko99}. It is possible to remove 
the sign definiteness assumptions on $V$ (as discussed in parts $\mathbf{(I)}$ and 
$\mathbf{(II)}$, see also \cite[Theorem\ 1.61]{Ko99} for $1 \leq n \leq 3$). Moreover, one should establish the connection with higher-order spectral shift functions (the Koplienko spectral shift function for Hilbert--Schmidt class perturbations $\cB_2(L^2(\bbR^n; d^nx))$ 
and its recent extension to $\cB_p(L^2(\bbR^n; d^nx))$-perturbations, $p\in \bbN$, 
$p \geq 3$, derived in \cite{PSS09} (see also \cite{AP11}). This lies beyond the scope of 
this paper and will be taken up elsewhere.    

\medskip

\noindent {\bf Acknowledgments.}
We are indebted to Barry Simon for helpful discussions. 


\end{document}